\documentclass[11pt,letter]{article}

\usepackage[
    top=1in,
    bottom=1in,
    left=.7in,
    right=1.3in
]{geometry}

\usepackage{latexsym}
\usepackage[T1]{fontenc}
\usepackage{titlefoot}
\usepackage[utf8]{inputenc}
\usepackage[english]{babel}
\usepackage{amscd}
\usepackage{amsmath}
\usepackage{amsfonts}
\usepackage{amssymb}
\usepackage{amsthm}
\usepackage{graphicx}
\usepackage{cite}
\usepackage{comment}
\usepackage{dsfont}
\usepackage{authblk}
\usepackage{mathtools,nccmath}
\usepackage{esint}
\usepackage{xfrac}
\usepackage[short labels]{enumitem}
\usepackage{dsfont}
\usepackage[unicode]{hyperref}
\usepackage{array}
 
\usepackage{hyperref}
\hypersetup{
    colorlinks = true,
    linkcolor = blue
    }
    
\usepackage[
        noabbrev,
        capitalise,
        nameinlink,
    ]
    {cleveref}

 \usepackage{xcolor}

\renewcommand{\epsilon}{\varepsilon}
\newcommand{\ep}{\epsilon}

\newcommand{\R}{\mathbb R}

\newcommand\dps{\displaystyle}
\newcommand\be{\begin{equation}}
\newcommand\ee{\end{equation}}
\newcommand\bea{\begin{eqnarray}}
\newcommand\eea{\end{eqnarray}}
\newcommand\beaa{\begin{eqnarray*}}
\newcommand\eeaa{\end{eqnarray*}}
\newcommand\bay{\begin{array}}
\newcommand\eay{\end{array}}
\newcommand\ba{\begin{align}}
\newcommand\ea{\end{align}}

\newcommand{\blue}[1]{{\color{blue}#1}}

\newtheorem{theorem}{Theorem}[section]
\newtheorem{lemma}[theorem]{Lemma}

\newtheorem*{claim*}{Claim}
\newtheorem{proposition}[theorem]{Proposition}

\theoremstyle{definition}

\newtheorem{remark}[theorem]{Remark}

\numberwithin{equation}{section}

\title{{On the logarithmic correction of transition fronts in shifting environments}}
\author{King-Yeung Lam$^1$ and Chang-Hong Wu$^{2,3}$
\\
{\footnotesize $^1$\textit{Department of Mathematics, The Ohio State University, Columbus, OH, 43210, USA} }\\
{\footnotesize $^{2}$\textit{Department of Applied Mathematics, National Yang Ming Chiao Tung University,\\ {\em Hsinchu, 300, Taiwan}} }\\
{\footnotesize $^{3}$\textit{ National Center for Theoretical Sciences, Taipei, 106, Taiwan}}

}
\date{}

\begin{document}

\maketitle

\begin{abstract}
In this paper, we investigate the location of the spreading front and convergence to traveling wave profile of solutions to the Fisher-KPP equation on the real line where the environment function is piecewise constant and has a moving jump discontinuity. 
{Key ingredients in our analysis include an extension of Bramson's seminal work \cite{bramson1983convergence} to Fisher--KPP equations with growing domains,
together with a gluing technique for constructing super- and subsolutions.
}
\end{abstract}

 \unmarkedfntext{Keywords: Logarithmic delay, shifting environment, reaction-diffusion equations, Fisher equation.\newline\hspace*{1.55em}2010 Mathematics Subject Classification: 35B40, 35K57, 92D25.\newline\hspace*{1.55em}E-mail: lam.184@osu.edu, changhong@math.nctu.edu.tw}

\section{Introduction}

In this paper, we consider the spreading dynamics of the following reaction-diffusion equation in a shifting environment:
\begin{equation}\label{eq0}
    \begin{cases}
        u_t = Du_{xx} + \Big(g(x-X(t))-u\Big)u &\text{ for }t>0,~x\in\mathbb{R},\\
        u(0,x) =u_0(x), &
    \end{cases}
\end{equation}
where $u(t,x)$ represents the population density of an invasive species, 
$D$ represents the diffusion rate of $u$; the intrinsic growth function $g(y)$ describes the environmental heterogeneity  and is positive and piecewise constant. 
The population described by equation \eqref{eq0} is affected by a {\it shifting environment}, which can arise from models of climate change \cite{potapov2004climate,li2014persistence} or pathogen spread \cite{Fang2016can}. It also serves as a toy model for the slower species among several tree species that are invading the North American continent after the last ice age \cite{shigesada1997biological}. Mathematically, when $X(t) = \beta t$ for some constant $\beta>0$, it becomes
a special case of the higher dimensional cylinder problem investigated by Hamel \cite{hamel1997reaction}.

We are interested in determining conditions on $g$ and $X(t)$ to characterize the existence and properties of a function $m(t)$ such that 
\beaa
u(t,y+m(t))\to \Phi(y)\quad  \text{ in } C_{loc}(\mathbb{R}),
\eeaa
where $\Phi$ denotes certain traveling wave solution satisfying, for some  positive constants $c$ and $R$, 
\begin{equation}\label{e.TWTW}
c\Phi' + D \Phi'' + (R - \Phi)\Phi=0 \quad \text{ in }\mathbb{R}, \quad \Phi(-\infty)=R, \quad \Phi(+\infty) = 0.
\end{equation}

\subsubsection*{Propagation in homogeneous environments}

The equation \eqref{eq0} belongs to the class of so-called Fisher-KPP equations and, in the case $g\equiv R$ is a constant, can be traced back to the pioneering work of Fisher \cite{fisher1937wave} and Kolmogorov, Petrovskii and Piskunov \cite{kolmogorov1937study} on spatial propagation phenomena. It has been proved that, when $g\equiv {R}$ is a constant, the equation in \eqref{eq0} admits a unique (up to translation) traveling wave solution of the form $u(t,x)=\Phi(x-ct)$ with some wave speed $c$ (such $\Phi$ are positive solutions of \eqref{e.TWTW}) if and only if $c\geq c_{\min}:=2\sqrt{DR}$. Moreover, $c_{\min}$ coincides with the spreading speed of the solutions \cite{aronson2006nonlinear}; that is, for each solution of \eqref{eq0} with compactly supported initial data, we have
\[
\begin{cases}
\lim\limits_{t\to\infty}\sup\limits_{|x|<ct} |u(t,x)-R|=0\qquad &\mbox{for any $0<c<c_{\min}$;}\\
\lim\limits_{t\to\infty}\sup\limits_{|x|>ct} |u(t,x)|=0\quad &\mbox{for any $c>c_{\min}$.}
\end{cases}\]
Consider next the level set of the solution $\xi_{b}(t):= \sup\{x \geq 0:~ u(t,x) \geq {b}\}$. Then 
the definition of spreading speed implies that  for each ${b} \in (0,{R})$,
$$
\xi_{b}(t) = c_{\min} t + o(t).
$$
Here we use big-oh and little-oh notation, meaning that $b_t = o(a_t)$ and $\tilde{b}_t = O(a_t)$ to mean that 
$$
\lim_{t\to\infty} \frac{b_t}{a_t}=0 \quad \text{ and }\quad \limsup_{t\to\infty} \left|\frac{\tilde{b}_t}{a_t}\right| < \infty.
$$

The precise long-time behavior of the solution has been established by Bramson \cite{bramson1983convergence} using a probabilistic approach. When the initial data $u_0:\mathbb{R} \to \mathbb{R}$ is supported in $(-\infty, x_0]$, he proved that  there exists a function $\gamma:[0,\infty) \to \mathbb{R}$ such that $|\gamma(t)| \leq O(1)$ and the solution $u$ of \eqref{eq0} satisfies
\bea \label{ea.0813.1}
\lim_{t\to\infty}\sup_{x\in\mathbb{R}}\Big|u(t,x)-\Phi_{c_{\min}}\Big(x-c_{\min}t+\frac{3}{2 \lambda_{\min}} \log t+\gamma(t)\Big)\Big|=0,
\eea
where $\lambda_{\min}=c_{\min}/{(2D)}$ and $\Phi_{c_{\min}}$ is the traveling wave solution with $c=c_{\min}$. In fact, it is shown that $\gamma(t) \to x_{\infty}$ for some $x_{\infty}$ depending on initial data \cite[Section 8, Theorem 3]{bramson1983convergence}. 
This result was also obtained by Hamel, Nolen, Roquejoffre, and Ryzhik \cite{hamel2013short,nolen2017convergence} using a PDE-based approach. 
For further discussion, we also refer to \cite{an2023pushed,an2023quantitative,bouin2020bramson,giletti2022monostable,hamel2016logarithmic,lau1985nonlinear,uchiyama1978behavior} and references cited therein. When the initial data is not compactly supported, then the solution may converge to one of the traveling wave solutions with supercritical speed. For example, it follows from \cite[Theorem B]{bramson1983convergence} that if 
\[{u_0(x) = x^qe^{-\lambda x},\ x>x_0 \qquad \text{ for some } \quad x_0>0,~q \in \mathbb{R},~\lambda \in (0, \lambda_{\min}),}
\]
%
then the exponential decay rate $\lambda$ selects a 
speed {$c= D\lambda + \tfrac{R}{\lambda}$} with the corresponding
traveling solution $\Phi_{c}$, i.e. there exists a bounded function $\gamma(t)$ such that
$$
\lim_{t\to\infty}
\sup_{x\in\mathbb{R}}
\Big|u(t,x)-\Phi_{c}
\Big(x-c t-\frac{q}{\lambda} \log t+ \gamma(t)\Big)\Big|=0.
$$
Note that, in contrast with \eqref{ea.0813.1}, the location of the level set and particularly the coefficient in the logarithmic correction depend on the fine decay property of the initial data in a subtle way, {and the logarithmic correction parameter $q/\lambda$ can be either positive or negative.} {We also refer to the recent work \cite{alfaro2025bramson,Zhang2025}, in which novel PDE arguments are given for  these results.}

Logarithmic corrections can also be observed in the co-invasion dynamics of competitive systems, which are sometimes complicated by the questions of linear determinacy \cite{hosono1998minimal,huang2011non,alhasanat2019minimal} and nonlocal pulling phenomenon \cite{girardin2019invasion,liu2019asymptotic}. In addition, the case where the faster species invades under a logarithmic correction, and the second invasion wave is of the pushed type, may arise in the competition system \cite{peng2021sharp,wu2023sharp}.

\subsubsection*{Precise asymptotics due to Bramson}

Let $w(t,x)$ be the solution to
\begin{equation}\label{e.0401.2}
\begin{cases}
    w_t = w_{xx} + f(w) \quad &\text{ for }(t,x) \in (0,\infty)\times \mathbb R\\
    w(0,x)=w_0(x) &\text{ in }\mathbb{R},
\end{cases}
\end{equation}
where $f\in C^1_{loc}([0,\infty))$ satisfies, for some constants $R,B>0$, 
\bea\label{f-cond}
\begin{cases}
f(s)<0 \quad \text{ for }s > B,\quad  f(s)>0\quad \text{ for }0 < s < B,\\
f(0) = f(B) = 0,\quad    f'(0) =R\geq f'(s) \quad \text{ for }0 < s \leq B,
\end{cases}
\eea
and that
there exists $0<{\rho}\leq 1$ 
such that \begin{equation}\label{e.a.423.11}
|R-f'(s)| \leq  \frac{1}{{\rho}}s^{{\rho}} \quad \text{ for all }s \in [0,2B].   
\end{equation}
(we can make {$\rho$} smaller if necessary in the proofs),
and $w_0$ satisfies
\begin{equation}\label{w-ic-q}
0\leq w_0 \in L^\infty(\mathbb R),\quad  \liminf_{x\to-\infty}w_0(x)>0\quad \text{ and } \quad 
w_0(x)= {x^q e^{-\lambda x}}\quad  \text{ for }x \geq x_0
 \end{equation}
 for some $q\in\R$, 
 {$\lambda>0$ and }
 $x_0 \geq 1$. 
 We collect the assumptions for this section here.
 \begin{description}
     \item[(F)]  Let $f$ satisfy \eqref{f-cond} and \eqref{e.a.423.11} for some constants $B,R>0$, 
     \item[(W0)] Let $w_0$ satisfy \eqref{w-ic-q} for some $q\in \mathbb{R}$, {$\lambda>0$} and $x_0\geq 1$.
 \end{description}


For each $\lambda>0$, we define 
\bea\label{c-lambda-2}
c_{\lambda}:=\lambda+\frac{R}{\lambda}\geq 2\sqrt{R},
\eea
where the equality holds if and only if $\lambda=\sqrt{R}$.
{We denote by $\Phi_{\lambda,R}$ a traveling wave solution satisfying\footnote{{We suppressed the dependence of $\Phi_{\lambda,R}$ on $B,f$ for clarify.}}
\begin{equation}\label{TW-B}
\begin{cases}
c_\lambda \Phi' + \Phi'' +f(\Phi) =0\quad \text{ for }z \in \mathbb{R},\\
\Phi(-\infty) = B, \qquad  \Phi(+\infty) = 0.
\end{cases}
\end{equation}
Since $\Phi_{\lambda,R}$ is unique up to translation, we choose its translation to absorb the constant phase shift in Bramson's asymptotics.
}
We recall the following two seminal results of Bramson. Note that $q$ can be positive or negative.

\begin{theorem}[{\cite[Theorem A and Example~1]{bramson1983convergence}}]\label{lem:bramson-pull}
Assume {\bf{(F)}}, {\bf{(W0)}} hold with corresponding constants $R>0$, $\lambda\in(0,\sqrt{R})$ and $q \in \mathbb{R}$.
Then the solution $w$ of \eqref{e.0401.2} with initial data $w_0$ satisfies
\begin{equation*}
{\sup_{ x \in \R}} ~ \big|w(t,x )  -\Phi_{\lambda,R}(x- m_{\lambda,q}(t)) \big|  \to 0\quad \text{ as }t\to\infty,
\end{equation*}
where $m_{\lambda,q}$ depends only on $\lambda$ and $q$, and 
\bea \label{m^w}
m_{\lambda,q}(t) := c_\lambda t + \frac{q}{\lambda }\log\left((c_\lambda - 2\lambda)t\right) \quad \text{ with }\quad c_{\lambda}=\lambda + \frac{R}{\lambda}.
\eea
\end{theorem}

\begin{theorem}[{\cite[Theorem A,  Example~3 and Example~4]{bramson1983convergence}}]
\label{lem:bramson-no-pull}
Assume {\bf{(F)}}, {\bf{(W0)}} hold with corresponding constants $R>0$, $\lambda=\sqrt{R}$, $q\in\mathbb R$.  
Let $w$ be the solution of \eqref{e.0401.2} with initial data $w_0$.
Then there exists a bounded function $\Lambda(t)$ such that
\begin{equation*}
{\sup_{ x\in\R}}  |w(t,x )  -\Phi_{\lambda,R}(x- {\tilde{m}_{q}}(t)-\Lambda(t)) |  \to 0\quad \text{ as }t\to\infty,
\end{equation*}
where \begin{equation}\label{m^w-2}
\tilde{m}_q(t) := 
\begin{cases}
c_{\min} t - \frac{3}{2\lambda_{\min}} \log t , &\text{ if }q< -2,\\
c_{\min} t - \frac{3}{2\lambda_{\min}} \log t + \frac{1}{\lambda_{\min}} \log \log t &\text{ if }q= -2,\\
c_{\min} t  + \frac{q-1}{2\lambda_{\min}} \log t &\text{ if }q > -2,
\end{cases}
\end{equation}
where $c_{\min}=2\sqrt{R}$ and $\lambda_{\min}=\sqrt{R}$.
\end{theorem}
We remark that $O(\log\log t)$ corrections for a class of Heaviside-like initial functions was also obtained by Uchiyama \cite{uchiyama1978behavior}.

\subsection{Propagation in a heterogeneous, shifting environments}\label{subsect:1.2}

The initial value problem \eqref{eq0} with $X(t)=\beta t$ was first introduced to study the impact of climate change on the persistence of species.
This class of models can be traced back to Potapov and Lewis \cite{potapov2004climate}, as well as Berestycki et al. \cite{berestycki2009can}. See also \cite{berestycki2008reaction, berestycki2009reaction} for extensions to higher spatial dimensions. 

The shifting environment also arise from another perspective in the work of Holzer and Scheel \cite{holzer2014accelerated}, where they examined a partially decoupled two-component system. In such a system, the faster species  generates an effective shifting environment for the slower species.
See also \cite{ducrot2021asymptotic,dong2021persistence,girardin2019invasion,liu2021stacked}  for related works on competition or prey-predator
systems in homogeneous environments.

The effect of imposing explicit shifting environments in the coefficients has also been 
considered in the contexts of competition systems \cite{yuan2019spatial} and problems with nonlocal diffusion \cite{li2018spatial,wu2019spatial}. See also the recent survey article \cite{wang2022recent} for recent literature on the subject.

The spreading dynamics of \eqref{eq0} was first considered by  Li et al. \cite{li2014persistence} under the assumption that $g$ is continuous and increasing with $g(-\infty)<0<g(+\infty)$. Alternative  conditions on $g$ are explored in \cite{hu2020spreading,Lam2022asymptotic}.  Particularly,
 a class of functional differential equations including \eqref{eq0} with $\sup_{\mathbb{R}}g\leq \max\{g(\pm\infty)\}$ was considered in  \cite{Lam2022asymptotic} with the Hamilton-Jacobi approach.

Particularly, when $g$ is strictly positive in $\mathbb{R}$, a new kind of spreading phenomena emerges. For instance, if $g$ is strictly increasing, an initially compactly supported population can spread at a supercritical speed $2\sqrt{g(-\infty)}<c_*<2\sqrt{g(+\infty)}$, that is not captured by the intuition from the limiting homogeneous environments as $x \to \pm \infty$. In this case, $c_*$ is referred to as the non-local pulling speed \cite{girardin2019invasion,holzer2014accelerated}. On the other hand, if $\sup_{\mathbb{R}}g>\max\{g(\pm\infty)\}$, the equation of \eqref{eq0} admits forced traveling wave solutions that propagate in locked step with the environmental shifting (see \cite{berestycki2018forced,holzer2014accelerated}).  
{The notion of forced waves studied in \cite{berestycki2018forced} is closely related to an earlier work of Hamel et al. \cite{hamel1997reaction}, where higher-dimensional reaction-diffusion problems in media without translation invariance are considered.}
Recently, an analytical framework is developed to give a unified treatment connecting the spreading result for monotone or non-monotonic environmental profile $g$ \cite{lam2024asymptoticspreadingkppreactive} which is based on the newly introduced notion of flux-limited solutions of Hamilton-Jacobi equations \cite{Imbert2017quasi}. An alternative approach based on comparison principle can be found in \cite{girardin2024spreading}.

It can therefore be said that the spreading property to \eqref{eq0} is by now well understood at the level of the spreading speed, i.e. the level set can be determined up to an error of order $o(t)$ across general settings. Our goal here is to establish the precise logarithmic correction of the level set to control the error to $O(1)$, and to show in addition the convergence of {the} solution to the traveling wave profile.
Precisely, we focus on the following problem:
\begin{equation}\label{main-eq0}
    \begin{cases}
        u_t = Du_{xx} + u(r(t,x)-u) &\text{ for }t>0,~x\in\mathbb{R},\\
        u(0,x) = u_0(x), &
    \end{cases}
\end{equation}
where the intrinsic growth rate $r$ depends on the parameters $R_+,R_->0$ and ${\beta}, \eta\in \mathbb{R}$, and is given by
\beaa
r(t,x):=
\begin{cases}
    R_+ &\text{ for }x > X(t),\\
    R_- &\text{ for }x \leq X(t),
\end{cases}
\eeaa
i.e., the environment is piecewise constant and has a jump discontinuity at $x=X(t)$, which is defined throughout this paper as
\bea\label{X_eta}
X(t)=\beta t -\eta \log (t+1),\quad t\geq0.
\eea
Particularly, the case $\eta\neq0$ is motivated by the co-invasion of competing species in homogeneous environments \cite{holzer2014accelerated,girardin2019invasion,liu2021stacked}. Specifically, the faster species has been shown to exhibit the logarithmic delay of the form $\frac{3}{2\lambda}\log t$ {\cite{peng2021sharp,wu2023sharp}}, so the effect on the invasion of the slower species, which is ``chasing behind", 
can be approximated by the scalar equation \eqref{main-eq0}.


We are interested in the dynamics, including the nonlocal pulling phenomenon, which arises when {$R_+>R_-$}. By non-dimensionalization, we may assume $D=1$, $R_+ =1$ and $R_- = 1-a$ for some $a \geq 0$.
Therefore, \eqref{main-eq0} recdues to 
\begin{equation}\label{main-eq}
    \begin{cases}
        u_t = u_{xx} + u(1-a\chi_{(-\infty,X(t)]}-u) &\text{ for }t>0,~x\in\mathbb{R},\\
        u(0,x) = u_0(x), &
    \end{cases}
\end{equation}
where $X(t)=X_{\beta,\eta}(t)$ is defined in \eqref{X_eta} and $u_0$ satisfies $0\leq u_0\leq1$ and 
\bea \label{e.bcc}
\liminf_{x\to-\infty}u_0(x)>0, \quad u_0(x)=0 \quad \text{ for all } x\geq x_0. 
\eea
for some $x_0>0$.

For given 
{$\beta\in\mathbb{R}$,}
$\eta\in \mathbb{R}$ and $0 < a < 1$, 
it is well known from \cite{holzer2014accelerated,Lam2022asymptotic}  that the spreading speed $c_*$ of \eqref{main-eq} exists and can be determined exactly as follows:
\bea\label{c*}
{c}_* =\begin{cases}
2 &\text{ if } \beta \le 2,\\
\dps 
\lambda_*+\frac{1-a}{\lambda_*}
&\text{ if } 2 < \beta < 2(\sqrt a +\sqrt{1-a}),\\
2\sqrt{1-a}  &\text{ if }  \beta \geq 2(\sqrt a +\sqrt{1-a}),
\end{cases}
\eea
where the exponent $\lambda_*\in (1-\sqrt{a}, \sqrt{1-a})$ is independent of $\eta$ and is given by
\bea\label{lambda*}
 \lambda_*=\frac{\beta}{2}-\sqrt{a} .
 \eea

Indeed, for the nonlocal pulling case ($2< \beta < 2(\sqrt a + \sqrt{1-a})$), the solution $u$ to \eqref{main-eq} behaves differently in the regions in front and behind the shifting discontinuity $X(t)$. Roughly speaking,
\begin{equation}
    u(t,x) \approx \begin{cases}
        \frac{1}{\sqrt{4\pi t}} e^{- \tfrac{x^2}{4t} + t + o(t)}  & \text{ for }x > X(t) + \delta t,\\
        e^{ - \lambda_*(x-c_* t) + o(t)}& \text{ for }c_*t + \delta t <x <  X(t) - \delta t .
    \end{cases}
\end{equation}
Thus the selection of supercritical spreading speed $c_*$ is understood via the selection of an {\it effective exponent} $\lambda_*$. This can be achieved via 
constructing super/subsolutions 
\cite{girardin2019invasion} or via the uniqueness of the limiting Hamilton-Jacobi equations \cite{Lam2022asymptotic,lam2024asymptoticspreadingkppreactive} satisfied by the rate function 
$$
w(t,x)=\lim_{\ep \to 0} - \ep \log u\left( \tfrac{t}{\ep}, \tfrac{x}{\ep}\right).
$$
However, the latter approach cannot capture fine asymptotics such as  Bramson's logarithmic correction. 

\subsubsection*{Our analytical approach}
In this paper, we will derive sharp estimates of $u(t,X(t))$ by constructions of sharp super- and subsolutions, using particularly the idea of Hamel et al. \cite{hamel2013short}, in connection with the solution of the linear heat equation in 
{the evolving}
domain $\{(t,x):~ x > X(t)\}$ ahead of the shifting discontinuity. Such super/subsolutions can then be glued to different solutions to the homogeneous problem in {the evolving}
domain $\{(t,x):~ x < X(t)\}$ behind the shifting discontinuity. It is essential that the environment is piecewise constant for our arguments to be applicable.

For the remainder of the article,  we fix an arbitrary constant 
{$b>0$}
and denote the location of the invasion front by
\bea\label{location-u}
\xi_{b} (t):= \sup\{x \geq 0:~ u(t,x) \geq {b}\}.
\eea
In view of the formula \eqref{c*}, we will discuss the following cases of $X(t) = \beta t - \eta \log (1+t)$:

\begin{itemize}
{
\item[{\bf(A0)}] (Spreading ahead of environment) $(\beta,\eta) \in  (-\infty,2)\times \mathbb{R}$ or $(\beta,\eta) \in \{2\}\times [\tfrac32, \infty)$.

    \item[{\bf(A1)}] (Supercritical pulling) 
   $(\beta,\eta) \in  (2,2(\sqrt{a} + \sqrt{1-a}))\times \mathbb{R}$ or $(\beta,\eta) \in \{2\}\times (-\infty, 3/2)$. 
   }
    \item[{\bf(A2)}]  (Critical pulling) $\beta =  2(\sqrt{a} + \sqrt{1-a})$. 
    \item[{\bf(A3)}] (No pulling) $\beta> 2(\sqrt{a}+\sqrt{1-a})$,
\end{itemize}


To investigate the location of the invasion front of $u(t,x)$, we relate it to the behavior of solutions of the KPP equations with a moving boundary, as described below.

\subsubsection*{Main results for the single equation with shifting environments.}

In this subsection, we fix $0<a<1$ in \eqref{main-eq} and let $u$ be a solution of \eqref{main-eq}-\eqref{e.bcc}.  We investigate the dependence of the front location of the solution to \eqref{main-eq} on the shifting  discontinuity 
\bea \label{e.X(t)}
X(t):=\beta t-\eta \log (t+1), \quad t\geq0
\eea
for some 
{$\beta \in \mathbb{R}$}
and $\eta \in\R$. 

In this subsection, we continue to use the notation $\Phi_{\lambda,R}$ for the traveling wave solution associated with the logistic nonlinearity $f(s)=s(R-s)$ for either $R = 1-a$ or $1$.
Let $R>0$ and $\lambda\in(0,\sqrt{R}]$. Set $c_{\lambda}:=\lambda+R/\lambda$.
We denote by 
$\Phi_{\lambda,R}$ the traveling wave solution satisfying
\begin{equation}\label{TW0}
\begin{cases}
c_\lambda \Phi' + \Phi'' +\Phi(R-\Phi) =0\quad \text{ for }z \in \mathbb{R},\\
\Phi(-\infty) = R, \qquad  \Phi(+\infty) = 0,
\end{cases}
\end{equation}
with a suitable spatial translation such that
\bea\label{TW-AS0}
\begin{cases}
\lim\limits_{z \to +\infty} e^{\lambda z}\Phi(z) = 1,& \quad \text{ if } c_{\lambda}>2\sqrt{R}\quad (\lambda\in(0,\sqrt{R})),\\
\lim\limits_{z \to +\infty} z^{-1}e^{\lambda z}\Phi(z) = 1,& \quad \text{ if } c_{\lambda}=2\sqrt{R}=:c_{\min} 
\quad (\lambda=\sqrt{R}).
\end{cases}
\eea
For convenience, we write
\[
\Phi_{\min,R}:=\Phi_{\sqrt R,R}.
\]

First, we state the case {\bf(A0)}, when the population keeps up with or surpasses the shifting environment, where the classical result  of \cite{bramson1983convergence} suggests that 
$$
\xi_{b}(t) = 2t - \frac{3}{2} \log t + O(1).
$$
However, the situation is subtle when $(\beta,\eta) =(2,\tfrac32)$, since then  the level set can be in front of or behind the environmental shift $X(t)$.

\begin{theorem}\label{thm:standard}
Assume $X(t) = \beta t - \eta \log (1+t)$ with either
\begin{itemize}
\item[{\rm(i)}] $\beta<2$; or 
\item[{\rm(ii)}] $\beta=2$ and $\eta \in [3/2,\infty)$.
\end{itemize}
Then for each $b \in (0,1)$, the level set $\xi_{b}(t)$ defined in \eqref{location-u} satisfies
\begin{equation}\label{standard-level-set}
 \xi_{b}(t)=2 t - \frac{3}{2}\log t+O(1),
\end{equation}
Moreover, if $\beta<2$ or $\beta =2$ and $\eta > \tfrac32$, then there exists a bounded function $\Lambda(t)$ such that for any $\delta>0$, 
\bea\label{standard-cov}
\sup_{x \geq X(t) + \delta \log t }\Big|u(t,x) - \Phi_{\min,1} \Big(x-2 t +\frac{3}{2}  \log t+\Lambda(t)\Big)\Big| \to 0 \text{ as } t\to\infty.
\eea
where $\Phi_{\min,1}$ is a critical traveling  wave\footnote{In the critical case $(\beta,\eta) = (2,\tfrac32)$, then we expect the solution to converge to a forced wave $\Psi$ satisfying
\[
\begin{cases}
2\Psi' + \Psi'' + \Psi(1- a\chi_{(-\infty,0]} - \Psi) = 0 &\text{ for }s \in \mathbb{R},\\
\Psi(-\infty) = 1-a, \quad \text{ and }\quad \Psi(+\infty) =0.\end{cases}
\]
}  satisfying \eqref{TW0}.
%
%
\end{theorem}

Next, we state the result in the supercritical pulling case {\bf (A1)}.

\begin{theorem}\label{thm:pulling}
Assume $X(t) = \beta t - \eta \log (1+t)$ with either
\begin{itemize}
\item[{\rm(i)}] $\beta \in \big(2,\,2(\sqrt a +\sqrt{1-a})\big)$ and $\eta \in \mathbb{R}$; or 
\item[{\rm(ii)}] $\beta=2$ and $\eta \in (-\infty,1/2)$.
\end{itemize}
Then there exists a bounded function $\Lambda(t)$ such that 
\bea\label{thm1.5-cov}
\sup_{x\in \R }\Big|u(t,x) - \Phi_{\lambda_*,1-a}\Big(x-c_* t +\frac{1}{\lambda_*}\left(\frac{3}{2} - \sqrt{a}\eta\right) \log t+\Lambda(t)\Big)\Big| \to 0 \text{ as } t\to\infty.
\eea
In particular, for each $b \in (0,1-a)$, the level set $\xi_{b}(t)$ defined in \eqref{location-u} satisfies
\begin{equation}\label{level-set}
 \xi_{b}(t)=c_* t - \frac{1}{\lambda_*}\left(\frac{3}{2} - \sqrt{a}\eta\right) \log t+O(1),
\end{equation}
where
$\Phi_{\lambda_*,1-a}(x)$ satisfies \eqref{TW0}
with speed $c_* = \lambda_* + (1-a)/\lambda_*$ and $\lambda_*$ is defined in \eqref{lambda*}. 
\end{theorem}
{The two cases (i) and (ii) in Theorem~\ref{thm:pulling} will be proved separately. The case (ii) is more delicate since $\xi_b(t) = X(t) - O(\log t)$.}  
In view of \eqref{lambda*}, we see that \eqref{level-set} can be expressed as 
\beaa
    \xi_{b}(t) = c_* t - \left(\eta +\frac{3-\beta\eta}{2\lambda_*}\right) \log t+O(1)
    = c_* t - \frac{1}{\lambda_*}\left[\frac{3}{2} - \left(\frac{\beta}{2} -\lambda_* \right)\eta\right] \log t+O(1).
\eeaa




{
\begin{remark}
For the remaining case $(\beta,\eta) \in \{2\}\times [\tfrac12, \tfrac32)$, we conjecture that \eqref{level-set} continues to hold. 
The main difficulty is to obtain a sharp lower-bound estimate as in
Lemma~\ref{l.a.0423.1}.
\end{remark}
}

Next, we state the result in the critical pulling case {\bf (A2)}.

\begin{theorem}\label{thm:cpulling}
Assume $X(t) = \beta t - \eta \log (1+t)$  for some  $\beta =2(\sqrt a +\sqrt{1-a})$ and $\eta \in \mathbb R$.
Set $q=-\tfrac32 + \eta \sqrt a$. Then there exists a bounded function $\Lambda(t)$ such that 
\beaa
\sup_{x\in \R }\Big|u(t,x) - \Phi_{\min,1-a}\Big(x-\tilde{m}_q(t)+\Lambda(t)\Big)\Big| \to 0 \text{ as } t\to\infty,
\eeaa
where  $\tilde{m}_q(t)$ is given in \eqref{m^w-2}.
The function  $\Phi_{\min,1-a}(x)$ {satisfies}
\eqref{TW0} 
with critical 
speed $c_{\min}= 2\sqrt{1-a}$ and $\lambda_{\min} = \sqrt{1-a}$. 
In particular,
\begin{equation*}
 \xi_{b}(t)=\tilde{m}_q(t)+O(1).
\end{equation*}
\end{theorem}

\medskip

The next result concerns the non-pulling case {\bf (A3)}.

\begin{theorem}\label{thm:no-pull}
Assume $X(t) = \beta t - \eta\log(1+t)$ for some  $\beta \in (2(\sqrt a +\sqrt{1-a}),\infty)$ and $\eta \in \mathbb R$.
Then there exists a bounded function $\Lambda(t)$ such that 
\begin{equation}\label{thm1.7-conv}
\sup_{x\in\R}  \Big|u(t,x )  -\Phi_{\min,1-a}\Big(x- c_{\min} t+ \frac{3}{2\lambda_{\min}}\log t + \Lambda(t)\Big) \Big|  \to 0\quad \text{ as }t\to\infty,
\end{equation} 
where $\Phi_{\min,1-a}(x)$ satisfies \eqref{TW0} 
with speed $c_{\min}:=2\sqrt{1-a}$ and $\lambda_{\min}:=\sqrt{1-a}$.
 In particular,
\begin{equation*}
 \xi_{b}(t)=c_{\min} t - \frac{3}{2\lambda_{\min}}\log t+O(1).
\end{equation*}
\end{theorem}

Since the level set asymptotics here is the same as the homogeneous KPP equation \eqref{e.0401.2} with $R=1-a$, it follows that the distant favorable region $\{(t,x):~x \geq X(t)\}$ can only change the dynamics of the level set $\xi_b(t)$ by at most a uniformly bounded quantity in case {\bf(A3)}.



\subsection{Formula of the level set $\xi_b(t)$ in case $\eta=0$.}\label{rk:new-thing}
We make an observation with $\eta = 0$, i.e. $X(t)  =\beta t$ for some {$\beta\in\mathbb{R}$}. Then the problem \eqref{main-eq} reduces to
\begin{equation}\label{main-eq2}
        u_t = u_{xx} + u(1-a\chi_{(-\infty,\beta t]}-u) \quad \text{ for }t>0,~x\in\mathbb{R},\\
\end{equation}
with initial data satisfying \eqref{e.bcc}.
Applying 
{Theorem~\ref{thm:standard},}
Theorem~\ref{thm:pulling}, Theorem~\ref{thm:cpulling}, Theorem~\ref{thm:no-pull}, we have
$$
\xi_{b}(t) = \begin{cases}
{2 t -\frac{3}{2}\log t+ O(1)} & {\text{ if } \beta<2,} \\ 
{2 t -\frac{3}{2(1-\sqrt{a})}\log t+ O(1)} & {\text{ if } \beta=2,} \\ 
    c_* t - \frac{3}{2\lambda_*}\log t + O(1)  &\text{ if }2 < \beta <2(\sqrt a + \sqrt{1-a}),\\
    c_{\min}t - {\frac{5}{4 \lambda_{\min}} \log t}+ O(1)&\text{ if }\beta = 2(\sqrt a + \sqrt{1-a}),\\
    c_{\min}t - \frac{3}{2 \lambda_{\min}} \log t+ O(1)  &\text{ if }\beta > 2(\sqrt a + \sqrt{1-a}).
\end{cases}
$$
where 
$$
\lambda_* = \frac{\beta}{2}-\sqrt{a}, \quad c_* =  \lambda_* + \frac{1-a}{\lambda_*},\quad \lambda_{\min}=\sqrt{1-a},\quad c_{\min} = \lambda_{\min} + \frac{1-a}{\lambda_{\min}}.
$$

It is interesting to observe the appearance of the factor $\tfrac54$, and that, as $\beta \searrow 2$, 
$$
\lambda_* = \frac{\beta}{2}-\sqrt a \to 1-\sqrt{a}, \quad\qquad c_* \to 2 \quad \text{ but }\quad \frac3{2\lambda_*} \not\to \frac32.
$$
For the readers' convenience, we summarize the results and outstanding problems for the spreading speed and a logarithmic delay in Tables~\ref{table1} and \ref{table2}. The open problem for the case $a<0$ and $\beta \geq 2\sqrt{1-a}$  (in which $r(t,x)$ is decreasing in $x$; see last row of Table \ref{table2}) is connected with \cite[Theorem 1.5]{berestycki2018forced} which says $u(t,x + \beta t) \to 0$ locally uniformly. Hence, this implies that $$
\theta^*:=\lim_{t\to\infty}\frac{\xi_{b}(t)-c_* t}{\log t}<0, \qquad \text{ provided that }\quad \beta = 2\sqrt{1-a}.$$ 

\subsection{Discussion}\label{sec:discussion}

In our main results we obtain sharp estimates of the location of the spreading front to a reaction-diffusion equation 
in the whole real line while being subjected to a shifting environment, where the shifting function $X(t)=X(t;\beta,\eta):=\beta t-\eta \log(t+1)$ incorporates a logarithmic delay (Subsections~\ref{subsect:1.2} and \ref{rk:new-thing}).
Our main results uncover a novel logarithmic delay phenomenon ({Theorems~\ref{thm:standard}}, \ref{thm:pulling}, \ref{thm:cpulling}, and \ref{thm:no-pull}) when the initial data {$u_0$ has support contained in $(-\infty,x_0]$.}

As we will see, a key ingredient underlying these results is the observation that the problem can be linked to a KPP equation on a growing domain {(Section~\ref{sec:thm1})}. By delicate analysis of this associated problem, motivated by the seminal work of Bramson \cite{bramson1983convergence}, we establish {Proposition~\ref{thm1}}, as a natural extension of Bramson's results. This observation deepens our understanding of the logarithmic delay phenomenon and enables us to apply this framework to a broader class of problems involving shifting environments.

{We should mention that the assumption of a piecewise constant environment is crucial in our analysis.
It is natural to ask whether analogous results hold for $r(t,x)=g(x-X(t))$, where $g$ is a smooth increasing function satisfying
\[
    g(-\infty)=1-a,\quad g(+\infty)=1,
\]
and converging exponentially to these limits. In such a case, we expect that this perturbation introduces only an $O(1)$ error, and leaves our logarithmic correction unchanged. 

However, 
a smooth environmental function no longer allows the equation to be decomposed exactly into homogeneous Fisher--KPP equations on the two sides of the moving interface. Hence, our present argument cannot be applied directly.
New ideas would be needed to control the contribution of the transition region. 
}



\begin{table}[h] 
    \centering
        \begin{tabular}{|c|c|c|c|}
        \hline
        $a$ & $\beta$  & $\dps c_*:=\lim_{t\to\infty}\frac{\xi_{b}(t)}{t}$ & $\dps \theta^*:=\lim_{t\to\infty}\frac{\xi_{b}(t)-c_* t}{\log t}$\\
        \hline
        $a=0$    &  N/A   & $2$ \cite{fisher1937wave,kolmogorov1937study}   & $\dps-\tfrac{3}{2
        \lambda_*}$ \cite{bramson1983convergence,hamel2013short}   \\
        \hline
         $0<a<1$    & $\beta<2$   & $2$ \cite{holzer2014accelerated,Lam2022asymptotic}   & $\dps -\tfrac{3}{2
        \lambda_*}$ \cite{bramson1983convergence,hamel2013short}   \\
         \hline
         $0<a<1$    & $\beta=2$   & $2$ \cite{holzer2014accelerated,Lam2022asymptotic}   &  {$\dps\tfrac{3}{2(1-\sqrt{a})}$ (Theorem~\ref{thm:pulling})} \\
        \hline
        $0<a<1$   & $2< \beta<2(\sqrt{a}+\sqrt{1-a})$   & $c_{nlp}$ \cite{holzer2014accelerated,Lam2022asymptotic}   & $\dps -\tfrac{3}{2
        \lambda_{nlp}}$ (Theorem~\ref{thm:pulling})   \\
        \hline
         $0<a<1$    & {$\beta=2(\sqrt{a}+\sqrt{1-a})$}   & $2\sqrt{1-a}$ \cite{holzer2014accelerated,Lam2022asymptotic}    & {$\displaystyle -\tfrac{5}{4\lambda_{\min}}$ (Theorem~\ref{thm:cpulling})}   \\
         \hline
          $0<a<1$    & {$\beta>2(\sqrt{a}+\sqrt{1-a})$}   & $2\sqrt{1-a}$ \cite{holzer2014accelerated,Lam2022asymptotic}    & {$\dps -\tfrac{3}{2\sqrt{1-a}}$ (Theorem~\ref{thm:no-pull})} \\ 
        \hline \hline 
         $a>1$    & $\beta<2$  & $2$ \cite{li2014persistence} & $\dps -\tfrac{3}{2
        \lambda_*}$ \cite{bramson1983convergence,hamel2013short} \\
         \hline
        $a>1$    & $\beta>2$  &  (extinction) \cite{li2014persistence}  & (extinction) \\
        \hline
    \end{tabular}
    \caption{The spreading dynamics to \eqref{main-eq2} when $X(t) = \beta t$, where the shifting function is increasing. Here $\lambda_* = 1$ and $\lambda_{nlp}$ correspond to the exponential decay rate of the traveling wave solutions with speeds $c_*=2$ and $c_{nlp}$, respectively.}\label{table1}
\end{table}

\begin{table}[h] 
    \centering
        \begin{tabular}{|c|c|c|c|}
        \hline
        $a$ & $\beta$  & $\dps c_*:=\lim_{t\to\infty}\frac{\xi_{b}(t)}{t}$ & $\dps \theta^*:=\lim_{t\to\infty}\frac{\xi_{b}(t)-c_* t}{\log t}$\\
        \hline
        $a<0$    &  $\beta<2$   & $2$ \cite{holzer2014accelerated,Lam2022asymptotic}  & $\dps -\tfrac{3}{2
        \lambda_*}$ \cite{bramson1983convergence,hamel2013short}  \\
        \hline
         $a<0$    & $2\leq \beta<2\sqrt{1-a}$   & $\beta$  \cite{berestycki2018forced,holzer2014accelerated,Lam2022asymptotic} &  0 \cite[Theorem 1.5(iii)]{berestycki2018forced}   \\
        \hline
        $a<0$    & $\beta\geq 2\sqrt{1-a}$   & $2\sqrt{1-a}$ \cite{holzer2014accelerated,Lam2022asymptotic}  &  {\bf Open problem}
        \\
        \hline
    \end{tabular}
    \caption{The spreading dynamics to \eqref{main-eq2} when $X(t) = \beta t$, where the shifting function is decreasing.}\label{table2}
\end{table}


\subsection{Organization of the paper}

The organization of this paper is as follows.
{In Section~\ref{sec:bramson} and Section~\ref{sec:2b}, we provide sharp estimates based on the work of Bramson \cite{bramson1983convergence} and Lau \cite{lau1985nonlinear}, which serve as key ingredients in the {proofs of our main results. In Section~\ref{sec:thm1}, we study a problem with a growing domain, extending Bramson's results for the Fisher--KPP equation on the real line to the setting of growing domains.
}
In Section~\ref{sec:4}, we establish a series of estimates relevant to studying the problem \eqref{main-eq} with a {shifting} environment and then we prove {Theorem~\ref{thm:standard},}
Theorem~\ref{thm:pulling}, Theorem~\ref{thm:cpulling}, and Theorem~\ref{thm:no-pull}. 
Finally, the Appendix contains the proof of  Lemma~\ref{e.a.lem.1} stated in
Section~\ref{sec:4}.
}





\section{{Estimates for the linear heat equation on the real line}}\label{sec:bramson}

In this section, we provide sharp estimates for the KPP equation with initial data behaving like $x^qe^{-\lambda x}$, {inspired by the seminal works of Bramson~\cite{bramson1983convergence} and Lau~\cite{lau1985nonlinear}.}


{To establish estimates for $w(t,x)$, we begin with analysis of the linear problem.}
Let $R\geq 0$ and let 
$\psi(x,t;w_0)$ be the solution to the linear problem
\begin{equation}\label{e.heat}
    \begin{cases}
        \psi_t = \psi_{xx} + R\psi &\text{ for }t>0,~ x \in \mathbb{R},\\
        \psi(0,x)=w_0(x) &\text{ in }\mathbb{R}.
    \end{cases}
\end{equation}
Then
\begin{equation}
    \psi(t,x) = e^{Rt} \int_{\mathbb{R}} w_0(y) \frac{1}{\sqrt{4\pi t}} e^{-\tfrac{(x-y)^2}{4t}}\,dy.
\end{equation}

Let $w$ denote the solution to the nonlinear equation \eqref{e.0401.2} whose initial data coincides with $\psi$. 
In Section \ref{sec:2b}, we will prove that for each $\delta>0$,  
$$
\sup_{x \geq m(t) + 2\delta t} \left|\frac{w(t,x)}{\psi(t,x)}  -1\right| \to 0\qquad \text{ as }t\to\infty.
$$
where 
\bea\label{m^w-def}
{m(t)}:=\begin{cases}
m_{\lambda,q}(t),& \quad \text{ if } \lambda\in(0,\sqrt{R}),\\
\tilde{m}_{q}(t),& \quad \text{ if } \lambda=\sqrt{R},
\end{cases}
\eea
with $m_{\lambda,q}(t)$ and $\tilde{m}_{q}(t)$ being defined in \eqref{m^w} and in \eqref{m^w-2}, respectively.

{Hereafter, we write $\text{sgn}\,q\in \{-1,0,1\}$ for the sign of a real number $q$, 
with the convention $\text{sgn}\,0=0$. For $a,b \in \mathbb{R}$, we denote 
\[
a \vee b := \max\{a,b\}, \qquad a \wedge b := \min\{a,b\}.
\]
}

\subsection{Refined estimates via heat kernel representation}

The following estimate is a variation of \cite[Lemma 4.1]{bramson1983convergence} under a different  set of assumptions.

\begin{lemma}\label{lem:A1}
Let $\bar\lambda >0$ be given. 
Assume that 
\begin{equation}\label{A1}
w_0\in L^{\infty} \quad \text{ and }\quad 
{w_0(y) \leq e^{-\bar\lambda y} \quad \text{ for all }y\geq y_0} 
\end{equation}
for some $y_0 \in \mathbb R$.
Then for each $\delta \geq 0$ and some constant $C$,
\begin{equation}\label{A2}
\log \left[\int_{J^c} w_0(y) \frac{1}{\sqrt{4\pi t}} e^{-\frac{(x-y)^2}{4t}} \,dy\right] \leq - \bar\lambda x + \bar{\lambda}^2 t 
- \frac{\delta^2}{8} t + C 
\end{equation} 
for all $x \in \mathbb{R}$ and  $t > 0$ and $J^c$ denotes the complement of
\begin{equation}
J = (x-2\bar\lambda t -\delta t, x-2\bar\lambda t + \delta t).
\end{equation}
\end{lemma}
\begin{proof} 
For $t \geq 1$, it follows from \cite[Lemma 4.1, p. 54]{bramson1983convergence} that
\begin{equation}\label{e.0423.1}
\log\left[ \int_{J^c} w_0(y) \frac{1}{\sqrt{4\pi t}} e^{-\frac{(x-y)^2}{4t}} \,dy \right] \leq -\bar\lambda x + \bar\lambda^2 t - \log(t\wedge 1) - \frac{\delta^2}{8} t     +C.
\end{equation}
This proves the assertion \eqref{A2} for $t\geq 1$.

For $0 \leq t \leq 1$, we use the following rough estimate:
$$
\int_{J^c} w_0(y) \frac{1}{\sqrt{4\pi t}} e^{-\frac{(x-y)^2}{4t}} \,dy \leq K\int_{-\infty}^\infty e^{-\bar\lambda y} \frac{1}{\sqrt{4\pi t}} e^{-\frac{(x-y)^2}{4t}} \,dy = K e^{-\bar\lambda x + \bar\lambda^2 t},
$$
where $K>1$ is chosen such that $w_0(y)\leq K  e^{-\bar\lambda y}$ for all $y\in\R$.
Hence, 
$$
\log\left[ \int_{J^c} w_0(y) \frac{1}{\sqrt{4\pi t}} e^{-\frac{(x-y)^2}{4t}} \,dy \right]\leq -\bar\lambda x+ \bar\lambda^2 t +\log K\quad \text{ for }t \in [0,1],~x \in \mathbb{R}.
$$
The lemma follows upon combining with \eqref{e.0423.1}.
\end{proof}

\begin{remark}
Note that {for any given $\lambda>0$  and $R>0$, }
\bea\label{identity1}
{e^{Rt}} \int_{\mathbb{R}} e^{-\lambda y} \frac{1}{\sqrt{4\pi t}} e^{-\frac{(x-y)^2}{4t}} \,dy=e^{-\lambda(x-c_\lambda t)} \quad \text{ where }~c_\lambda = \lambda + \frac{R}{\lambda}.
\eea

It follows from \eqref{identity1} and Lemma \ref{lem:A1} (taking $\bar\lambda = \lambda$) that for any given $\delta>0$, there exists $C_0>0$ such that
\begin{align}
e^{Rt}\int_{J} e^{-\lambda y} \frac{1}{\sqrt{4\pi t}} e^{-\frac{(x-y)^2}{4t}} \,dy &=  e^{Rt}\Big(\int_{\mathbb{R}} -\int_{J^c}\Big)e^{-\lambda y} \frac{1}{\sqrt{4\pi t}} e^{-\frac{(x-y)^2}{4t}} \,dy \notag \\
&\geq e^{-\lambda(x-c_\lambda t)}(1-C_0e^{-\frac{\delta^2}{8}t})
\label{e.0323.3}
\end{align}
for $x \in \mathbb{R}$ and $t>0$.
\end{remark}

Next, we present the main technical lemma, which is a special case of \cite[Lemma 4.1]{bramson1983convergence}, and provides upper and lower estimates for solution $\psi$ to the linear equation \eqref{e.heat}  with a specific initial data $w_0(x)$. 
\begin{lemma}\label{lem:Bramson-0813}
{Let $R>0$, $\lambda>0$, $q\in\R$ and $x_0\geq 1$ be given}, and let $\psi$ denote the solution to \eqref{e.heat} with initial condition $w_0$, which  satisfies {\bf(W0)}.
Then for each $\delta\in (0,2\lambda)$ and each sufficiently small 
$\epsilon \in (0,\delta/6)$, there exists $C_1>1$ such that the following statements hold.
\begin{itemize}
   \item[{\rm(i)}] In $\{(t,x):~ x \geq (2\lambda + \delta)t+x_0,~t>0\}$, we have
 \begin{equation}\label{e.0813.2}
  \psi(t,x) \geq 
 (x-2\lambda t - (\text{sgn}\,q)\delta t )^q \Big(1-C_1e^{-\frac{\delta^2}{8}t}\Big)e^{-\lambda(x-c_\lambda t)}
\end{equation} 
    \item[{\rm(ii)}] In $\{(t,x):~ x \geq (2\lambda + \delta)t+x_0,~t>0\}$, we have
    \begin{equation}\label{e.0813.1a}
 \psi(t,x) \leq (x-2\lambda t + (\text{sgn}\,q)\delta t) ^{q} e^{-\lambda(x-c_\lambda t)} +  C_1e^{-(\lambda - \epsilon)(x-c_\lambda t) - \frac{\delta^2}{19}t}  
\end{equation}
   \item[{\rm(iii)}] In $\{(t,x):~ x \geq c_\lambda t ,~t>0\}$, we have
      \begin{equation}\label{e.0813.1bb}
 \psi(t,x) \leq C_1\,(x\vee 1)^{|q|}\, e^{-(\lambda-\ep)(x-c_\lambda t)}.
\end{equation}
    \item[{\rm(iv)}] 
In $\{(t,x):~ x \in \mathbb R,~t>0\}$, we have
   \begin{equation}\label{e.0813.1b}
 \psi(t,x) \wedge 1 \leq C_1\, (x\vee 1)^{|q|}\, e^{-(\lambda-\ep)(x-c_\lambda t)}.
\end{equation}
 
\end{itemize}
Here we recall that $c_\lambda = \lambda + \tfrac{R}{\lambda}$.
\end{lemma}

\begin{remark}\label{rmk.815.1}
Fix $s \geq 0$ and let $\psi$ be a solution of \eqref{e.heat} in the domain $[s,\infty)\times \mathbb R$. If $\psi(s,\cdot)$ satisfies {\bf(W0)} for some fixed $s\geq 0$, then by \eqref{e.0813.1bb}, we have
$$
\psi(t,x) \leq C_1(x\vee 1)^{|q|} e^{-(\lambda-\ep)(x-c_\lambda (t-s))} \quad \text{for }t>s,~x\in\mathbb R.
$$
\end{remark}
\begin{remark}\label{rmk:0423.1b}
We state an observation for later purposes: 
For $\beta > 2\lambda$ and $\eta\in \mathbb{R}$,  by setting $x = \beta t - \eta \log (t+1)$ in
 \eqref{e.0813.2} and \eqref{e.0813.1a}, we have
\begin{equation}
\lim_{t\to\infty}    \frac{\psi(t, \beta t - \eta \log t)}{\psi(t, 1-M + \beta t - \eta \log t)} = e^{-\lambda (M-1)} \quad \text{ for each }M \in \mathbb{R}.
\end{equation}
\end{remark}

\begin{proof}[Proof of Lemma~\ref{lem:Bramson-0813}]
Fix $q \in \mathbb{R}$, 
{$\lambda>0$, 
$x_0\geq1$}, and let $$J = (x-2\lambda t -\delta t, x-2\lambda t + \delta t).$$
First, we derive \eqref{e.0813.2}.
Using the standard heat kernel representation, and \eqref{w-ic-q}, 
\begin{align}
\psi(t,x)  &\geq e^{Rt}\int_{J}w_0(y) \frac{1}{\sqrt{4\pi t}} e^{-\frac{(x-y)^2}{4t}}\,dy\notag \\
&= e^{Rt}\int_{J}y^q e^{-\lambda y} \frac{1}{\sqrt{4\pi t}} e^{-\frac{(x-y)^2}{4t}}\,dy\notag \\
& \geq    [(x-2\lambda t - (\text{sgn}\,q)\delta t )]^q  e^{Rt}\int_{J} e^{-\lambda y} \frac{1}{\sqrt{4\pi t}} e^{-\frac{(x-y)^2}{4t}} \,dy \notag\\ 
&\geq  [(x-2\lambda t - (\text{sgn}\,q)\delta t )]^q (1-C_0e^{-\frac{\delta^2}{8}t})e^{-\lambda(x-c_\lambda t)}
\label{e.0813.5a}
\end{align}
for all $t>0$ and  $x \geq (2\lambda +\delta)t +x_0$, 
where we used \eqref{e.0323.3} for the last inequality.
This proves \eqref{e.0813.2}.

Next, we prove the inequality \eqref{e.0813.1a}. 
Set $\bar\lambda = \lambda - \ep$ for some $\ep \in (0,\tfrac{\delta}6)$, we note that $J = (x-2\lambda t- \delta t,x -2\lambda t + \delta t)$ and $\bar J = (x-2\bar\lambda t - \tfrac23 \delta t, x- 2\bar\lambda t + \tfrac23 \delta t)$ satisfy
$$
 J^c\subset \bar{J}^c.
$$
Hence, we use Lemma \ref{lem:A1} to estimate the integral on $J^c$ as
\begin{equation}
\label{e.0813.6a}
\log \left[e^{Rt} \int_{J^c}w_0(y) \frac{1}{\sqrt{4\pi t}} e^{-\frac{(x-y)^2}{4t}} \right] \leq - (\lambda-\ep)(x- c_{\lambda-\ep} t) 
- \frac{(\tfrac{2}{3}\delta)^2}{8}t + C 
\end{equation}
 for all $x \in \mathbb{R}$ and  $t > 0$.     
By taking $\ep\in (0,\tfrac{\delta}6)$ small enough, we may assume 
\beaa
(\lambda-\ep)(c_{\lambda-\ep} - c_\lambda) < \frac{\delta^2}{18\cdot 19},
\eeaa 
so that \eqref{e.0813.6a} implies
\begin{equation}\label{e.0813.6}
\log \left[e^{Rt} \int_{J^c}w_0(y) \frac{1}{\sqrt{4\pi t}} e^{-\frac{(x-y)^2}{4t}} \right] \leq - (\lambda-\ep)(x- c_{\lambda} t) - \frac{\delta^2}{19}t + C 
\end{equation}
for all $x \in \mathbb{R}$ and  $t >0$.

Similar as in \eqref{e.0813.5a},  we have, for all $t>0$ and  $x \geq (2\lambda +\delta)t +x_0$,
\begin{align}
{e^{Rt}} \int_{J}w_0(y) \frac{1}{\sqrt{4\pi t}} e^{-\frac{(x-y)^2}{4t}}\,dy &\leq 
{e^{Rt}} 
{(x-2\lambda t + (\text{sgn}\,q)\delta t )^q}
\int_{J} e^{-\lambda y} \frac{1}{\sqrt{4\pi t}} e^{-\frac{(x-y)^2}{4t}} \,dy \notag \\
 &\leq  {e^{Rt}} 
 {(x-2\lambda t + (\text{sgn}\,q)\delta t )^q}\int_{\mathbb{R}} e^{-\lambda y} \frac{1}{\sqrt{4\pi t}} e^{-\frac{(x-y)^2}{4t}} \,dy \notag  \\
 &=  
{[(x-2\lambda t + (\text{sgn}\,q)\delta t )\vee 1]^q}e^{-\lambda(x-c_\lambda t)} 
 \label{e.0813.4}
\end{align}
where we used \eqref{identity1}. We can then conclude \eqref{e.0813.1a} by combining \eqref{e.0813.6} and \eqref{e.0813.4}. 

To prove  
\eqref{e.0813.1bb} and \eqref{e.0813.1b}, we choose {$C_0>1$} such that 
$$
w_0(y) \leq C_0(y\vee 1)^{|q|} e^{-\lambda y} \quad \text{ holds for all }y \in \mathbb{R}.
$$
which is a consequence of \eqref{w-ic-q}. One can then observe that the argument for \eqref{e.0813.4} and $\delta \in (0,2\lambda)$ yield
\begin{align}  
{e^{Rt}} \int_{J}w_0(y) \frac{1}{\sqrt{4\pi t}} e^{-\frac{(x-y)^2}{4t}}\,dy &\leq C_0{[(x-2\lambda t + \delta t )\vee 1]^{|q|}}e^{-\lambda(x-c_\lambda t)} \notag\\
&\leq C_0{(x\vee 1)^{|q|}}e^{-{\lambda(x-c_\lambda t)}}
\label{e.0813.4b}
\end{align}
for all $t >0$ and $x \in\mathbb R$.
{Combining with \eqref{e.0813.6}},
we have
\begin{equation}
 \psi(t,x) \leq   C_0{(x\vee 1)^{|q|}} \left(e^{-(\lambda-\ep)(x-c_\lambda t)} + e^{-\lambda(x-c_\lambda t)}  \right) \quad \text{ for }t>0,~x\in\mathbb{R},
\end{equation}
which implies \eqref{e.0813.1bb}. 

To derive \eqref{e.0813.1b}, we observe that for $x\geq c_{\lambda}t$,
\eqref{e.0813.1bb} immediately implies \eqref{e.0813.1b}.
For $x< c_{\lambda}t$, we have 
\beaa
 C_1\, (x\vee 1)^{|q|}\, e^{-(\lambda-\ep)(x-c_\lambda t)}\geq C_1 (x\vee 1)^{|q|} >1 \geq \psi \wedge 1.
\eeaa
Therefore, \eqref{e.0813.1b} holds.
The proof is now complete.
\end{proof}

\section{The linear versus nonlinear equations}\label{sec:2b}
Next, we state the refined asymptotic properties connecting the solutions of the linear and nonlinear equations with the same, front-like initial data. To ease the burden of notation, observe that we may assume without loss of generality that $R=1$, in view of the change of variables $(t',x')=(Rt,\sqrt{R} x)$. Henceforth, we will state the result for $R>0$, but only prove the case $R=1$.

\begin{proposition}\label{lem:w-estimate}
Let $f$ and $w_0$ satisfy {\bf (F)} and {\bf(W0)} for some given constants $R>0$, $q \in \mathbb R$ and $\lambda \in (0,\sqrt R]$. 
Let $w$ and $\psi$ denote the solution to \eqref{e.0401.2} and \eqref{e.heat}, respectively, with the same initial condition $w_0$. 
Then for each $\delta>0$, there exist $\delta_1>0$, $T>0$ and $C>0$ such that 
\bea\label{psi-w-est}
1-C e^{-\delta_1 t} \leq \frac{w(t,z+m (t))}{\psi(t,z + m(t))} \leq 1 \qquad \text{ in }\{(t,x):~t\geq T,~z \geq 2\delta t\},
\eea
where $m(t)$ is specified in \eqref{m^w-def}.
{Moreover,} for any $0<{k}<1$ and $X(t)$ such that $(c_\lambda + {k}) t \leq X(t) \leq (c_\lambda + 1/{k})t$,  
\begin{equation}\label{e.a.0423.4}
  \lim_{t\to\infty} \left[(x-2\lambda t)^{-q} e^{\lambda x - (\lambda^2 +R)t} w(t,x)\right]_{x=X(t)} =1,
\end{equation}
and 
\begin{equation}\label{w-upper-lower}
{\bar C}^{-1} t^{q} e^{-\lambda(x-c_\lambda t)} \leq w(t,x) \leq {\bar C} t^{q}e^{-\lambda(x-c_\lambda t)}
\end{equation}
in $\{(t,x):~ (c_\lambda + {k})t \leq x \leq (c_\lambda + 1/{k})t,\, t \geq 1\}$, where $\bar C>1$ is a constant.
\end{proposition}

\begin{remark}\label{rmk:0423.1}
We state an observation for later purposes: For each $\beta > c_\lambda$,  $\eta \in \mathbb{R}$, by setting $X(t) = \beta t - \eta \log (t+1)$ in \eqref{e.a.0423.4}, we have 
\begin{equation}
\lim_{t\to\infty}    \frac{w(t, \beta t - \eta \log t)}{w(t, 1-M + \beta t - \eta \log t)} = e^{-\lambda (M-1)} \quad \text{ for each }M \in \mathbb{R}.
\end{equation}
\end{remark}

To prove Proposition~\ref{lem:w-estimate}, 
define the remainder function
$$
E(t,x):= \psi(t,x) - w(t,x).
$$
Since it satisfies $E_t - E_{xx} +RE = Rw - f(w)$, it follows from the Duhamel’s formula that
\begin{equation}\label{E}
E(t,x)=    \int_0^t  e^{R(t-s)}\int_\mathbb{R} \frac{1}{\sqrt{4\pi (t-s)}} e^{ -\frac{(x-y)^2}{4(t-s)}}
{[Rw(s,y)-f(w(s,y))]}
\,dy ds.
\end{equation}
It is clear that $E\geq 0$ due to $R=f'(0)$ and $f(w) \leq Rw$, thanks to \eqref{f-cond}. It follows that
\begin{equation}\label{upper-bd0413}
w(t,z + m(t)) \leq \psi(t,z + m(t)) \quad \text{ for all }t>0,~z \in \mathbb{R}.
\end{equation}
To show the lower bound for $w/\psi$, we need to estimate $E(t,x)$ from above. 

\begin{lemma}\label{lem:E}
Given sufficiently small $\delta>0$ and $\sigma>0$, 
there exist $C>0$ and $T>1$ such that
\begin{equation}\label{e.a.0422.1}
0\leq E(t,z+m(t)) \leq C 
e^{- (1+\sigma)\lambda z}
\end{equation}
for $z \geq 2\delta t$ and  $t \geq T$, where $m(t)$ is specified in \eqref{m^w-def}.
\end{lemma}
\begin{proof}

Consider the equation \eqref{e.0401.2} and let $f$ be a function satisfying \eqref{f-cond} and \eqref{e.a.423.11} for some $R=f'(0)>0$, $B>0$ and $0<\rho \leq 1$. Again, we may assume without loss that $R=1$. 
Henceforth, we fix an initial data $w_0$ satisfying \eqref{w-ic-q} for some $q \in \mathbb{R}$ and $\lambda \in (0,1]$ and $x_0\geq 1$, and let $w$ (resp. $\psi$) be the solution of \eqref{e.0401.2} (resp. of \eqref{e.heat} with $R=1$).  By the maximum principle, it is easy to see that
\bea\label{e.0814.1}
w(t,x) \leq B \wedge \psi(t,x) \leq (1+B) (1\wedge \psi(t,x) )\quad \text{ for }t>0,~x\in \mathbb{R}.
\eea
We fix $\delta \in (0,2\lambda)$ and a small $\sigma\in(0,(\rho + \max\{1-\sigma,0\})/3)$ such that
\begin{equation}\label{e.815.sig}
    \frac9{\lambda}\sigma^2 <  \sigma \lambda \delta.
\end{equation}
Having chosen $\sigma$, we then fix $\ep>0$ small enough so that
\begin{equation}\label{e.815.ep}
    \frac9{\lambda}\sigma^2 + \ep c_\lambda <  \sigma \lambda \delta \quad \text{ and }\quad \lambda_1:= (1+3\sigma)(\lambda - \ep)  \geq (1+2\sigma)\lambda + \ep.
\end{equation}
Note also that ${0}<\lambda <\lambda_1 <1$ in case $\lambda <1$. 
 Next, note that there is $K>0$ such that
$$
0\leq 
    {w(s,y)-f(w(s,y))}
    \leq {K}|w(s,y)|^{1+{\rho}}
    \leq  K(1+B)^{1+\rho} |1 \wedge \psi(s,y)|^{1+3\sigma}
$$
for all $t>0$ and $x\in \R$, where the second inequality is due to \eqref{e.0814.1} and $0<3\sigma<\rho$.  Hence, 
thanks to \eqref{e.0813.1b} in Lemma~\ref{lem:Bramson-0813} and the definition of $\lambda_1$, we have
\begin{align*}
    0&\leq 
    {w(s,y)-f(w(s,y))}
\leq  {J(s,y)}  \quad \text{ for }s>0,~y \in \mathbb R,
\end{align*}
where we set $\bar q = |q|(1+{3}\sigma)$, 
\bea\label{e.0814.3}
J(s,y)=
 K' G(s) \min\{1,   [y\vee 1]^{\bar q}  e^{-\lambda_1 y}\}  \quad \text{ for }s>0,~y \in \mathbb R.
\eea
for some $K'>0$ and $G(s) = e^{{\lambda_1}c_{\lambda}s}$.
Note that $G(s) >1$ for all $s \geq 0$.

Recalling the definition of $E(t,x)$ in \eqref{E} and using \eqref{e.0814.3}, we have 
\begin{equation}\label{e.0814.2}
    0\leq E(t,x)\leq E_1(t,x):= \int_0^t  e^{(t-s)}\int_\mathbb{R} \frac{1}{\sqrt{4\pi (t-s)}} e^{ -\frac{(x-y)^2}{4(t-s)}}   {J}(s,y)\,dy, \quad t>0,~x\in \mathbb R. 
\end{equation} 
Observe that, for each fixed $s>0$, the integrand 
$$
\psi(t,x) = e^{(t-s)}\int_\mathbb{R} \frac{1}{\sqrt{4\pi (t-s)}} e^{ -\frac{(x-y)^2}{4(t-s)}}  {J}(s,y)\,dy
$$is a solution to the heat equation in the interval $(s,\infty) \times \mathbb{R}$ with initial data 
$$
\psi(s,y)\equiv {J}(s,y) {=} K' G(s) \min\big\{1,   (y\vee 1)^{\bar q}  e^{- \lambda_1 y}\big\}.
$$
Let 
$$
\Omega_1:=\{(t',s',x'):~x'\geq (c_\lambda + \delta)(t'-s') ,~0<s'<t'\}.
$$
By Remark \ref{rmk.815.1}, there exists $C_1>0$ such that for $(t,s,x) \in \Omega_1$,
\begin{align}
&e^{(t-s)} \int_\mathbb{R} \frac{1}{\sqrt{4\pi (t-s)}} e^{ -\frac{(x-y)^2}{4(t-s)}} J_1(s,y)dy\notag \\
&= K'G(s) \int_\mathbb{R} \frac{1}{\sqrt{4\pi (t-s)}} e^{ -\frac{(x-y)^2}{4(t-s)}}  \min\left\{ 1,~ (y\vee 1)^{\bar q}  e^{-\lambda_1 y} \right\}dy\notag \\
&\leq K'G(s)\cdot C_1 (x\vee 1)^{\bar q}
\exp\Big(
-(\lambda_1 -\ep)[x-c_{\lambda_1}(t-s) ]\Big) \notag \\
&\leq K'\cdot C_1 (x \vee 1)^{\bar q}
\exp\Big(-(\lambda_1 -\ep)[x-c_{\lambda_1}(t-s) - c_{\lambda}s]+ \ep c_\lambda s\Big). 
\label{e.815.2}
 \end{align}
Next, define 
$$
\Omega_2:= \{(t',x'):~x'\geq (c_\lambda + \delta)t',~t'>0\}.$$ Then for each $(t,x) \in \Omega_2$, we have $(t,s,x) \in \Omega_1$ for all $s \in [0,t]$ so that we may
integrate \eqref{e.815.2} over $s\in[0,t]$, to get $C_2>0$ such that
\begin{equation}\label{E1estimate0812}
0\leq E_1(t,x) 
\leq C_2  \,(x\vee 1)^{\bar q} \exp\Big(-(\lambda_1-\ep)(x-(c_\lambda \vee c_{\lambda_1}) t) + \ep c_\lambda t\Big)  \quad \text{ for }(t,x)\in \Omega_2.
\end{equation}
Note that $c_{\lambda_1}>c_{\lambda}$ if $\lambda =1$ and $c_{\lambda_1}<c_{\lambda}$ if $\lambda <1$. Recall the formula of $m(t)$ given in \eqref{m^w-def}, for different cases of $\lambda \in (0,1]$ and $q\in \mathbb R$. 
Since for each $t \gg 1$ and $z>2\delta t$, $(t,x)=(t,z+m(t)) \in \Omega_2$, we may substitute $x= z+m(t)$ in \eqref{E1estimate0812} to obtain
\begin{align}
&E_1(t,z+m(t))\notag \\ &\leq C_3(z + t)^{\bar q} \exp\Big( - (\lambda_1 - \ep)z + (\lambda_1 - \ep)|m(t) - (c_\lambda \vee c_{\lambda_1})t| + \ep c_\lambda t \Big) \notag\\
&\leq C_3 (z+t)^{\bar q}\exp\Big(-(1+2\sigma)\lambda z  + \lambda_1 \big|m(t) - (c_\lambda \vee c_{\lambda_1})t \big| + \ep c_\lambda t\Big) \notag \\
&\leq C_3 e^{-(1+\sigma)\lambda z} (z+t)^{\bar q} \exp\Big( - \frac{\sigma \lambda}{2}  z- \sigma \lambda \delta t  + \lambda_1 \big|{m(t)} - (c_\lambda \vee c_{\lambda_1})t\big| + \ep c_\lambda t\Big) 
\label{e.815.4}
\end{align}
for $t>1$ and $z \geq 2\delta t$, where we used \eqref{e.815.ep} in the first inequality. 

It remains to prove \eqref{e.a.0422.1} for the cases $\lambda <1$ and $\lambda = 1$ separately.

Consider the case $\lambda<1$. We have $c_\lambda \vee c_{\lambda_1} = c_{\lambda}$ as $0 < \lambda < \lambda_1 < 1$\footnote{We may adjust $\sigma$ smaller to ensure that.}. Then using $m(t) = c_\lambda t + o(t)$, we have (thanks to \eqref{e.815.ep})
\begin{align*}
-\sigma \lambda\delta t  + \lambda_1 \left|m(t) - (c_\lambda \vee c_{\lambda_1})t\right| + \ep c_\lambda t&= -\sigma \lambda\delta t + \lambda_1 \left|m(t) - c_\lambda t\right| + \ep c_\lambda t \\
&= -\sigma \lambda \delta t + o(t)+ \ep c_\lambda t \leq 0 \quad \qquad \text{ for }t\gg 1. 
\end{align*}
Hence, we can substitute the above into \eqref{e.815.4} to get
\begin{equation}\label{e.815.5}
E_1(t,z + m(t)) \leq C_3 e^{-(1+\sigma)\lambda z} (z + t)^{\bar q} e^{- {\sigma\lambda z}/{2}} \leq C_4 e^{-(1+\sigma)\lambda z} 
\end{equation}
for $t \gg 1$, $z \geq 2\delta t$.
Combining with \eqref{e.0814.2}, we derive \eqref{e.a.0422.1}.

Consider the case $\lambda=1$. then 
$$
\lambda  < \lambda_1 < 1+3\sigma, \quad  c_{\lambda_1} = \lambda_1 + \frac{1}{\lambda_1}\quad \text{ and }\quad c_\lambda = 2,$$ so that 
$$
| (c_\lambda\vee c_{\lambda_1})-c_\lambda| = |c_{\lambda_1}- c_\lambda | = \left| \lambda_1 + \frac{1}{\lambda_1} -2\right| = \frac{(\lambda_1 - 1)^2}{\lambda_1} \leq \frac{9 \sigma^2}{\lambda}.
$$
Using also \eqref{e.815.ep}, we may evaluate, for $t\gg 1$:
\begin{align}
-\sigma \lambda\delta t  + \lambda_1 \Big|m(t) - (c_\lambda \vee c_{\lambda_1})t\Big| + \ep c_\lambda t&= -\sigma \lambda\delta t + \lambda_1 \Big|c_\lambda t + o(t) - (c_\lambda \vee c_{\lambda_1}) t\Big| + \ep c_\lambda t \notag\\
&\leq -\sigma \lambda \delta t + \frac{9\sigma^2}{\lambda} t + o(t)+ \ep c_\lambda t \leq 0. 
\label{e.815.11}
\end{align}
Hence, we can substitute \eqref{e.815.11} into \eqref{e.815.4} to deduce again that \eqref{e.815.5} holds for $t\gg 1$ and $z \geq 2\delta t$. Combining with \eqref{e.0814.2}, we derive \eqref{e.a.0422.1}.
\end{proof}

We are in position to prove the main result of this section.
\begin{proof}[Proof of Proposition~\ref{lem:w-estimate}]
For ease of notation and without loss of generality, we treat the case $R=1$. 
Fix a small $\delta>0$.
Observe (and recall $m(t)$ from \eqref{m^w-def}) that for $z\geq 2\delta t$,
\beaa
\frac{w(t,m(t) + z)}{\psi(t,m(t) + z)}= 
\frac{\psi(t,m(t) + z)-E(t,m(t) + z)}{\psi(t,m(t) + z)}= 1- \frac{E(t,m(t) + z)}{\psi(t,m(t) + z)}.
\eeaa 
It suffices to estimate the quotient $E/\psi$. First, 
using \eqref{e.0813.2} in Lemma~\ref{lem:Bramson-0813}, 
there exists $C'>0$ such that all large $t$,
\begin{equation}\label{psi-est}
\psi(t,m(t) + z) \geq C'\left[\frac{(z + (c_\lambda - 2\lambda )t)^q}{{\gamma_{\lambda,q}(t)}}\right] e^{-\lambda z} \quad \text{ for }z \geq 2\delta t,
\end{equation}
where $\gamma_{\lambda,q}$ is defined by
\bea\label{gamma-q}
 \gamma_{\lambda,q}(t):=
\begin{cases}
t^q,&\quad \text{ if } \lambda\in(0,1),\\
t^{-3/2},&\quad \text{ if } \lambda=1,\ q<-2,\\
t^{-3/2}(\log t),&\quad \text{ if } \lambda=1,\ q=-2,\\
t^{(q-1)/2},&\quad \text{ if } \lambda=1,\ q>-2.
 \end{cases}
 \eea
By taking $\sigma>0$ small enough, we may apply Lemma \ref{lem:E} to get constants $T>0$ and $C_6>1$ such that 
\begin{align*}
 \frac{E(t,m(t) + z)}{\psi(t,m(t) + z)}  
&\leq \frac{ C {e^{- (1+\sigma)\lambda z}} 
}{C'\left[\frac{(z + (c_\lambda - 2\lambda )t)^q}{{\gamma_{\lambda,q}(t)}}\right] e^{-\lambda z}} \quad \text{ for }t \geq T,~ z \geq 2\delta t.
\end{align*}
 Hence, using $z \geq {2\delta t}$, there exists {$\delta''>0$ and } $C''>0$ such that 
\begin{equation}\label{e.0414.1} 
\sup_{z \geq 2\delta t}\frac{E(t,m(t) + z)}{\psi(t,m(t) + z)} \leq 
2C'' e^{-\delta'' t} \quad \text{ for }t\gg 1.
\end{equation}

Hence, using \eqref{e.0414.1}, it follows that 
for some $\delta_1>0$ we have 
\begin{align}\label{e.0414.2} 
\left.\frac{w(t,x)}{\psi(t,x)} \right|_{x=m(t) + z} =\left.\frac{\psi(t,x)-E(t,x)}{\psi(t,x)} \right|_{x=m(t) + z} \geq 1- \frac1{\delta_1}e^{-\delta_1 t}
\end{align}
holds uniformly in $\{(t,x):~ t \geq T,~ x \geq m(t)+ {2\delta t}\}$. Therefore, we obtain \eqref{psi-w-est}.

We now prove \eqref{e.a.0423.4}.
To see that, choose $\delta>0$ small enough so that 
{$2\lambda + 2\delta <c_{\lambda}+k$}
which implies that $X(t) > (2\lambda+ \delta)t + x_0$ for $t \gg 1$. This allows us to take $x=X(t)$ in \eqref{e.0813.1a} and \eqref{e.0813.2} in Lemma \ref{lem:Bramson-0813}.

Also, since $X(t)\leq (c_{\lambda}+1/{k})t$,
we can choose  $\epsilon=\epsilon(\delta)>0$ small enough to ensure
\begin{equation}\label{e.815.8}
    \ep(x- c_\lambda t)\leq \frac{\ep}{{k}}t  < \frac{\delta^2}{19}t \quad \text{ for }x = X(t),~t \gg 1.
\end{equation}
On the one hand, in case $\lambda \in (0,1)$, we apply \eqref{e.815.8} and \eqref{e.0813.1a} in Lemma~\ref{lem:Bramson-0813} to obtain
\begin{align}\label{ineq0423-1}
&\left[(x - 2\lambda t)^{-q} e^{\lambda x - (\lambda^2 +R)t} {\psi}(t,x)\right]_{x=X(t)}\leq \frac{(X(t)-2\lambda t + (\text{sgn}\,q)\delta t)^{q}}{(X(t)-2\lambda t )^q}[1+o(1)].
\end{align}
On the other hand, in case $\lambda =1$, we apply \eqref{e.0813.2} in Lemma~\ref{lem:Bramson-0813} to obtain
\begin{align}\label{ineq0423-2}
&\left[(x - 2\lambda t)^{-q} e^{\lambda x - (\lambda^2 +R)t} {\psi}(t,x)\right]_{x=X(t)}\geq \frac{(X(t)-2\lambda t - (\text{sgn}\,q)\delta t)^{q}}{(X(t)-2\lambda t )^q}[1+o(1)],
\end{align}
where $o(1)$ terms in \eqref{ineq0423-1} and \eqref{ineq0423-2} tend to zero as $t\to\infty$, for each fixed $\delta$ and $\ep$. Hence, by 
letting $t \to \infty$ and then $\delta \searrow 0$
and  combining with  \eqref{psi-w-est}, we obtain \eqref{e.a.0423.4}.

Finally, \eqref{w-upper-lower} follows by combining 
\eqref{e.0813.1a}, \eqref{e.0813.2},
and  \eqref{psi-w-est}. 
\end{proof}

{
\section{The Fisher-KPP equation with growing domain}\label{sec:thm1}
}

{
A key step in our analysis is to extend the classical results of Bramson concerning the homogeneous Fisher-KPP equations on the real line, to a large class of growing domains.

To this end, let $\Omega_{\zeta}$ be a growing domain given by
\beaa
\Omega_{\zeta}:=\{(t,x):\, t>\zeta(x), \ x\in\R\}, 
\eeaa
where $\zeta$ is a piecewise smooth function satisfying
\begin{equation}\label{zeta-def}
\zeta(x) = 0 \quad \text{ in }(-\infty,0] \quad \text{ and }\quad \zeta'(x)\geq0 \quad \text{ in }\mathbb R.
\end{equation}
For example, for each $k\geq 0$, the following describes the boundary of a sectorial region:
\beaa
\zeta(x)=
0 \quad \text{ if } x\leq 0, \quad \text{ and }\quad 
\zeta(x)=  kx \quad \text{ if } x> 0.
\eeaa
We consider the following KPP equation in a growing domain:
\bea\label{eq-KPP-B}
\begin{cases}
u_t=u_{xx}+f(u),&\quad (t,x)\in \Omega_{\zeta},\\
u(t,x)|_{t=\zeta(x)}=g(x),&\quad x\in\partial_P\Omega_{\zeta},
\end{cases}
\eea
{where} $f$ satisfies the assumption {\bf(F)} 
and $g\in L^{\infty}(\R)$ is a given function satisfying, for some constant $B>0$,
\beaa
0\leq g \leq B,\quad \liminf_{x\to-\infty}g(x)>0. 
\eeaa
For given constant $R>0$, the above includes the particular case of 
\bea\label{KPP-f}
f(u)=u (R- u),\quad B=R. 
\eea
}

{
Let us fix $\lambda\in(0,\sqrt{R}]$ and set $c_{\lambda}:=\lambda+R/\lambda$. For the growing domain
$\Omega_{\zeta}$, we assume that the right-most boundary point travels faster than $x=c_{\lambda} t$, so that the invasion front lags behind the expanding boundary. Precisely, we assume
\bea\label{zeta-cond}
(\exists \ep_0>0) \qquad 
{\ep_0\leq \zeta'(x) \leq  \frac{1}{c_\lambda+\ep_0} \qquad \text{ for }x \gg 1.}
\eea
{We are interested in the precise conditions on $(\lambda, \zeta(x), g(x))$ to determine the existence and properties of some function $m(t)$ such that}
\beaa
u(t,y+m(t))\to \Phi_{\lambda,R}(y)\quad  \text{ in } C_{loc}(\mathbb{R}),
\eeaa
where $\Phi_{\lambda,R}$ denotes the traveling wave solution that
satisfies \eqref{TW-B}
with a suitable spatial translation such that
\bea\label{TW-AS}
\begin{cases}
\lim\limits_{z \to +\infty} e^{\lambda z}\Phi(z) = 1,& \quad \text{ if } c_{\lambda}>2\sqrt{R}\quad (\lambda\in(0,\sqrt{R})),\\
\lim\limits_{z \to +\infty} \frac{1}{z}e^{\lambda z}\Phi(z) = 1,& \quad \text{ if } c_{\lambda}=2\sqrt{R}=:c_{\min} 
\quad (\lambda=\sqrt{R}),
\end{cases}
\eea
Hereafter, we denote $\lambda_{\min}:=\sqrt{R}$ and 
write $\Phi_{\min,R}=\Phi_{\lambda_{\min},R}$ to represent the minimal traveling front corresponding to speed $c_{\lambda}=c_{\min}$.
}

{Regarding the growing-domain problem, we have the following result, which can be viewed as
a growing-domain counterpart of Bramson's results for the Fisher--KPP equation on the real line.

\medskip

\begin{proposition}\label{thm1} 
Let $\lambda\in(0,\sqrt R]$ and $q\in\R$ be given.
Let $u(t,x)$ be the solution to \eqref{eq-KPP-B} in the growing domain $\Omega_{\zeta}$ where $\zeta(x)$ satisfies \eqref{zeta-def} and \eqref{zeta-cond}. Furthermore, assume that
\begin{align}
 &\lim_{x \to +\infty} \left[(x - 2\lambda t)^{-q} e^{\lambda x - (\lambda^2 + R)t} u(t,x)\right]_{t=\zeta(x)}=1.\label{e.0422.1}
\end{align}
Then the following holds.
\begin{itemize}
    \item[(i)] If $\lambda\in(0,\sqrt{R})$, then  
    \begin{equation}
\sup_{ {-\infty <} x \leq \zeta^{-1}(t)}  |u(t,x )  -\Phi_{\lambda,R}(x- {m_{\lambda,q}(t)}) |  \to 0\quad \text{ as }t\to\infty,
\end{equation}
where $\Phi_{\lambda,R}$ satisfies \eqref{TW-B} and \eqref{TW-AS}, and $m_{\lambda,q}(t)$ is given in \eqref{m^w}.  

\item[(ii)] If $\lambda=\sqrt{R}$, then  
there exists a bounded function $\Lambda(t)$ such that
\begin{equation}\label{conv-ii-thm1}
\sup_{{-\infty < } x\leq \zeta^{-1}(t)}  |u(t,x )  -\Phi_{\min,R}(x- {\tilde{m}_q(t)} - \Lambda(t)) |  \to 0\quad \text{ as }t\to\infty,
\end{equation}
where  $\Phi_{\min,R}$ satisfies \eqref{TW-B} and \eqref{TW-AS}  and $\tilde{m}_q(t)$ is given in \eqref{m^w-2}.
\end{itemize}
\end{proposition}
}

{
\begin{remark}
When $\Omega_{\zeta}=[0,\infty)\times \R$ is the entire upper half-plane, i.e. $\zeta(x)\equiv0$, then condition \eqref{e.0422.1} reduces to
\begin{align}
 &\lim_{x \to +\infty} x^{-q} e^{\lambda x} u(0,x)=1,\label{e.0422.2}
\end{align}
depending on the values of $\lambda\in(0,\sqrt{R}]$ and $q\in\R$, then
the conclusion of {Proposition~\ref{thm1}} has been established by Bramson \cite[pp. 6-8,~187,188]{bramson1983convergence}. We base our analysis upon extending these results.
\end{remark}
}

{
\begin{remark}
{Proposition~\ref{thm1}}(i) asserts that the $O(1)$ correction in the level set can be taken to be {exactly} $\frac{q}{\lambda}\log(c_\lambda - 2\lambda)$, with the understanding that the traveling wave is normalized by $e^{\lambda x}\Phi_{\lambda,R}(x)  \to 1$ as $x\to\infty$.
\end{remark}
}


{We prepare the proof {Proposition~\ref{thm1}} with two separate lemmas.}

\begin{lemma}\label{lem:limsup-case}
Let $u(t,x)$ be the solution to \eqref{eq-KPP-B} and $\zeta(x)$ be as given in \eqref{zeta-def} satisfying \eqref{zeta-cond}.
Assume further that 
\begin{align}
\limsup_{t \to +\infty} \left[(x - 2\lambda t)^{-q} e^{\lambda x - (\lambda^2 +R)t} u(t,x)\right]_{x=\zeta^{-1}(t)} \leq  1.\label{limsup}
\end{align}
holds for some $\lambda \in (0,\sqrt R]$ and $q \in \mathbb{R}$. Then the following holds.
\begin{itemize}
    \item[(i)] If $\lambda\in(0,\sqrt{R})$, then  
\begin{equation*}
\limsup_{t\to\infty }\Big[\sup_{ {-\infty} < x \leq \zeta^{-1}(t)}\Big(u(t,x )  -\Phi_{\lambda,R}(x- m_{\lambda,q}(t))\Big)\Big] \leq  0,
\end{equation*}
where $m_{\lambda,q}(t)$ is defined in \eqref{m^w}.
\item[(ii)] If $\lambda=\sqrt{R}$, then there exists a constant $C$ such that
\begin{equation*}
\limsup_{t\to\infty }\Big[\sup_{ {-\infty}< x \leq \zeta^{-1}(t)}\Big(u(t,x )  -\Phi_{\min,R}(x- \tilde{m}_{q}(t)-C)\Big)\Big] \leq  0,
\end{equation*}
 where $\tilde{m}_{q}(t)$ is defined in \eqref{m^w-2}.
\end{itemize}
\end{lemma}
\begin{proof}
We first deal with (i).
Let $q$ and $\lambda\in(0,\sqrt{R})$ be fixed. For each $M>0$ {sufficiently large}, we denote  
$\overline{w}^M(t,x)$ as the solution of \eqref{e.0401.2} with a smooth nonincreasing initial data such that
\beaa
\overline{w}^M(0,x)=
\begin{cases}
{M},& \quad x\leq M;\\
x^{q}e^{-\lambda x},& \quad x> M+1.
\end{cases}
\eeaa
By Theorem \ref{lem:bramson-pull}, it is clear that for any $M \in [2\max \{B, \|g\|_{L^\infty}\},{\infty})$, we have $\overline{w}^M(t,x) = \Phi_\lambda(x-m_{\lambda,q}(t)) + o(1)$ uniformly in $x$, as $t\to\infty$. Hence, 
it suffices to show that for each  $\ep'>0$, we can choose $M$ such that 
\begin{equation}\label{e.a.0422.10}
u(t,x) \leq \overline{w}^M(t,x-\ep')\quad \text{ in }\Omega_{\zeta}.
\end{equation}

{Let us fix an arbitrarily small $\delta\in(0,\ep_0/3)$, where $\ep_0>0$ is given in \eqref{zeta-cond}.}
By Proposition \ref{lem:w-estimate}, {for some $\delta_1>0$ and $C>0$,}
\begin{equation}\label{e.a.0422.5}
 {(1 - Ce^{-\delta_1t})\overline{\psi}^M (t,x)\leq \overline{w}^M (t,x) \leq \overline{\psi}^M(t,x)}
\end{equation}
{for $x\geq m_{\lambda,q}(t)+2\delta t$ and for all large $t$,}
where $\overline\psi^M$ is the solution the linear heat equation \eqref{e.heat} with initial data $\overline{w}^M(0,x)$.

Next, we claim that
\bea \label{e.a.0422.6}
\limsup_{t \to +\infty} \left[(x - 2\lambda t)^{-q} e^{\lambda x - (\lambda^2 +R)t}  \overline{\psi}^M (t,x)\right]_{x=\zeta^{-1}(t)}=1.
\eea
To see that, we use 
$c_\lambda + 2\delta < (\zeta^{-1})'(t) \leq {1}/{\ep_0}$ for $t\gg 1$ (by \eqref{zeta-cond}), so that we can choose  $\epsilon=\epsilon(\delta)>0$ small enough to ensure
$$
\ep(x- c_\lambda t) < \frac{\delta^2}{19}t \quad \text{ for }x = \zeta^{-1}(t),~t \gg 1.
$$
Hence, by  \eqref{e.0813.1a} of Lemma~\ref{lem:Bramson-0813}
\begin{align*}
&\left[(x - 2\lambda t)^{-q} e^{\lambda x - (\lambda^2 +R)t} \overline{\psi}^M(t,x)\right]_{x=\zeta^{-1}(t)}\leq \frac{(\zeta^{-1}(t)-2\lambda t + (\text{sgn}\,q)\delta t)^{q}}{(\zeta^{-1}(t)-2\lambda t )^q}[1+o(1)],
\end{align*}
where we also used $c_{\lambda}=\lambda+R/\lambda$.  Letting first $t \to \infty$ and then $\delta \searrow 0$, we deduce \eqref{e.a.0422.6}.
Finally, combining with  \eqref{e.a.0422.5} we obtain
\bea \label{e.a.0422.7}
\limsup_{t \to +\infty} \left[(x - 2\lambda t)^{-q} e^{\lambda x - (\lambda^2 +R)t} \overline{w}^M (t,x)\right]_{x=\zeta^{-1}(t)}=1.
\eea

Next, we compare $u$ with a translation of $\overline{w}^M$ over the parabolic boundary $\partial_P\Omega$.
By \eqref{e.0422.1} and \eqref{e.a.0422.7}, we see that for any $\ep'>0$, there exists $t_{\ep'}>0$  such that 
\bea \label{e.a.0422.9}
u(t,\zeta^{-1}(t))\leq \overline{w}^M(t,\zeta^{-1}(t)-\ep')\quad \text{ for all } t\geq t_{\ep'}.
\eea
Note that $t_{\ep'}$ is independent of $M$ large, since $\bar{w}^M$ is monotone increasing in $M$.
Since $u(t,x) \leq {\max\{B, \|g\|_{L^{\infty}}\}}$ for all $t>0$, $x \in \mathbb{R}$, 
and  $\{(t,x)\in \partial_P\Omega,~ t\in(0,t_{\ep}]\}$ is precompact,
we can choose $M=M({\ep'})\gg1$ such that 
\bea \label{e.a.0422.8}
u(t,x)\leq \overline{w}^M(t,x-\ep')\quad \text{ for all } (t,x)\in\{(t,x)\in \partial_P\Omega,~ t\in[0,t_{\ep'}]\}.
\eea
Having verified \eqref{e.a.0422.9} and \eqref{e.a.0422.8}, we can invoke comparison principle to deduce
\eqref{e.a.0422.10}, and completes the proof of part (i) of the theorem.

The proof of part (ii) follows the same argument as in (i), except that {Theorem}~\ref{lem:bramson-pull} is replaced by {Theorem}~\ref{lem:bramson-no-pull}. 
Note that Theorem~\ref{lem:bramson-no-pull} contains a O(1) correction.
This completes the proof.
\end{proof}

\begin{remark}\label{rk:3.1}
If \eqref{limsup} is relaxed as  
\begin{align*}
\limsup_{t \to +\infty} \left[(x - 2\lambda t)^{-q} e^{\lambda x - (\lambda^2 +R)t} u(t,x)\right]_{x=\zeta^{-1}(t)} \leq  k_0
\end{align*}
for some constant $k_0>0$, then the conclusion 
(ii) of Lemma~\ref{lem:limsup-case}  holds without change; while conclusion (i) becomes 
\begin{equation*}
\limsup_{t\to\infty }\Big[\sup_{ {-\infty}< x \leq \zeta^{-1}(t)}\Big(u(t,x )  -\Phi_{\lambda,R}(x- m_{\lambda,q}(t)-h)\Big)\Big] \leq  0
\end{equation*}
where $h = \frac{1}{\lambda}\log k_0$. Indeed, it suffices to replace
$\overline{w}^{M}(t,x)$ with 
$\overline{w}^{M}(t,x-h)$ in the proof.
\end{remark}



\begin{lemma}\label{lem:liminf-case}
Let $u(t,x)$ be the solution to \eqref{eq-KPP-B} and $\zeta(x)$ be as given in \eqref{zeta-def} satisfying \eqref{zeta-cond}.
Assume further that 
\begin{align}
\liminf_{t \to +\infty} \left[(x - 2\lambda t)^{-q} e^{\lambda x - (\lambda^2 +R)t} u(t,x)\right]_{x=\zeta^{-1}(t)} \geq  1.\label{liminf}
\end{align}
holds for some $\lambda \in (0,\sqrt R]$ and $q \in \mathbb{R}$. Then the following holds.
\begin{itemize}
    \item[(i)] If $\lambda\in(0,\sqrt{R})$, then  
\begin{equation*}
\liminf_{t\to\infty }\Big[\inf_{ {-\infty}< x \leq \zeta^{-1}(t)}\Big(u(t,x )  -\Phi_{\lambda,R}(x- m_{\lambda,q}(t))\Big)\Big] \geq  0,
\end{equation*}    
    where $m_{\lambda,q}$ is defined in \eqref{m^w}.
\item[(ii)] If $\lambda=\sqrt{R}$, then  
then there exists a constant $C$ such that
\begin{equation*}
\liminf_{t\to\infty }\Big[\inf_{{-\infty}< x \leq \zeta^{-1}(t)}\Big(u(t,x )  -\Phi_{\min,R}(x- \tilde{m}_q(t)-C)\Big)\Big] \geq  0,
\end{equation*}
where $\tilde{m}_q$ is defined in \eqref{m^w-2}.
\end{itemize}
\end{lemma}
\begin{proof}
We first deal with (i).
Let $q\in \mathbb R$ and $\lambda\in(0,\sqrt R)$ be fixed. For each $M>0$, we denote 
$\underline{w}^M(t,x)$ as the solution of \eqref{e.0401.2} with 
\beaa
\underline{w}^M(0,x)=\frac{1}{M}\chi_{(-\infty,M]}+x^q e^{-\lambda x}\chi_{[M,\infty)}.
\eeaa
In view of Theorem~\ref{lem:bramson-pull}, we see that for each $M>0$, $\underline{w}^M(t,z+m_{\lambda,q}(t))= \Phi_{\lambda,R}(z) + o(1)$  as $t\to\infty$ uniformly in $z$, and that the limit $\Phi_{\lambda,R}$ is independent of $M$. Therefore, it suffices to prove that for each $\ep'>0$, there exists $M>0$ such that 
\bea\label{goal-3.4}
u(t,x)\geq \underline{w}^M(t,x+\ep') \quad \text{ in }\{(t,x)\in \Omega:~ t \geq 1\}
\eea

Note that $\liminf\limits_{x\to-\infty}\underline{w}^M(0,x)>0$
and 
\bea\label{cov-0}
\text{ as }M\to\infty,\quad \underline{w}^M\to 0 \quad \text{ uniformly in }[0,L]\times(-\infty,L],
\eea
for any given $L>0$.
As in the proof of Lemma~\ref{lem:limsup-case}, we can derive 
\bea\label{limit-2}
\liminf_{t \to +\infty} \left[(x - 2\lambda t)^{-q} e^{\lambda x - (\lambda^2 +R)t} \underline{w}^M(t,x)\right]_{x=\zeta^{-1}(t)}= 1.
\eea

Next, let $\Omega' = \Omega\cap \{t\geq 1\}$. We will compare $u$ with a translation of $\underline{w}^M$ over the parabolic boundary 
of $\Omega'$.
By \eqref{e.0422.1} and
\eqref{limit-2}, we see that for any $\ep'>0$, there exists $t_{\ep'}>0$ (independent of $M$) such that 
\bea\label{P-bd-1}
u(t,\zeta^{-1}(t))\geq \underline{w}^M(t,\zeta^{-1}(t)+\ep')\quad \text{ for all } t\geq t_{\ep'}.
\eea

Next, since $\liminf\limits_{x\to-\infty}u_0(x)>0$, by the Harnack inequality, there exists $\kappa_0>0$ such that
$$
u(t,x) \geq 2\kappa_0 \quad \text{ in }\{(t,x)\in \partial_P \Omega': 1 \leq t \leq t_{\ep'}\}. 
$$
Thanks to \eqref{cov-0}, we can take $M=M_{\ep'}\gg1$ such that 
\bea\label{P-bd-2}
\underline{w}^M(t,x+\ep')<\kappa_0\leq u(t,x)
\quad 
\text{ in }\{(t,x)\in \partial_P \Omega': 1 \leq t \leq t_{\ep'}\}. 
\eea
Combining \eqref{P-bd-1} and \eqref{P-bd-2}, we can apply the comparison principle to conclude that 
\beaa
u(t,x)\geq \underline{w}^M(t,x+\ep')\quad \text{ for all } (t,x)\in\Omega\cap \{t\geq 1\}.
\eeaa
This proves \eqref{goal-3.4} and the proof of part (i) is complete.

The proof of part (ii) follows the same argument as in (i), except that Theorem~\ref{lem:bramson-pull} is replaced by  Theorem~\ref{lem:bramson-no-pull}. 
This completes the proof.
\end{proof}

\begin{remark}\label{rk:3.2}
Analogously to Remark~\ref{rk:3.1}, 
if \eqref{liminf} is relaxed as
\begin{align*}
\liminf_{t \to +\infty} \left[(x - 2\lambda t)^{-q} e^{\lambda x - (\lambda^2 +R)t} u(t,x)\right]_{x=\zeta^{-1}(t)} \geq  k_0
\end{align*}
for some constant $k_0>0$, then the conclusion (ii) of Lemma~\ref{lem:liminf-case} still holds, while (i) becomes 
\begin{equation*}
\liminf_{t\to\infty }\Big[\inf_{{-\infty}< x \leq \zeta^{-1}(t)}\Big(u(t,x )  -\Phi_{\lambda,R}(x- m_{\lambda,q}(t)-h)\Big)\Big] \geq  0
\end{equation*}
where $h = \frac{1}{\lambda}\log k_0$.
\end{remark}

We are ready to prove {Proposition~\ref{thm1}}.

\begin{proof}[Proof of {Proposition~\ref{thm1}}]
The proof of (i) follows directly by combining  Lemma~\ref{lem:limsup-case}~(i) and Lemma~\ref{lem:liminf-case}~(i). 
For part (ii), by combining Lemma~\ref{lem:limsup-case} (ii) and Lemma~\ref{lem:liminf-case}~(ii), 
{it follows that $\xi_{b}(t)=\tilde{m}_{q}(t)+O(1)$ as $t\to\infty$. Furthermore, for each $n\in\mathbb{N}$, consider 
\beaa
u_n(t,x)=u(t+t_n,x+\tilde{m}_{q}(t_n)),\quad t\in\R,\ -\infty<x+\tilde{m}_{q}(t_n)\leq \zeta^{-1}(t+t_n).
\eeaa
Standard compactness arguments imply that, up to the extraction of a subsequence, $u_n$ converges locally uniformly in $\R^2$ to a function $u_{\infty}$, which solves $u_t=u_{xx}+f(u)$ and satisfies 
$0\leq u_{\infty}\leq B$ in $\R^2$. Moreover, by again invoking Lemma~\ref{lem:limsup-case} (ii) and Lemma~\ref{lem:liminf-case}~(ii), we deduce that there exist two constants $C_1,C_2\in\R$ such that 
\beaa
\Phi_{\min,R}(x-c_{\min}t+C_2)\leq u_{\infty}(t,x)\leq \Phi_{\min,R}(x-c_{\min}t+C_2)\quad \text{ for all } (t,x)\in\R^2
\eeaa
Thanks to the Liouville-type result in 
\cite[Theorem 3.5]{beresycki2007generalized}, there exists $\tilde{C}\in\R$ such that
\beaa
 u_{\infty}(t,x)=\Phi_{\min,R}(x-c_{\min}t+\tilde{C})\quad \text{ for all } (t,x)\in\R^2.
\eeaa
In particular, this implies that $u_n(0,\cdot)\to \Phi_{\min,R}(\cdot+\tilde{C})$ as $n\to\infty$ in $C_{loc}(\R)$, where $\tilde{C}$ may depend on the choice of subsequence $\{t_n\}$. In addition, using the facts 
$\Phi_{\min,R}(-\infty)=B$, $\Phi_{\min,R}(+\infty)=0$ and
that
\begin{align*}
 &\liminf_{t\to\infty}\Big[\min_{0\leq x\leq \tilde{m}_q(t)-C} u(t,x)\Big]\to B\quad \text{ as } C\to\infty,\\   
 &\limsup_{t\to\infty}\Big[\max_{ x\geq \tilde{m}_q(t)+C} u(t,x)\Big]\to0\quad \text{ as } C\to\infty,
\end{align*}
we can obtain \eqref{conv-ii-thm1}. This completes the proof.
}
\end{proof}

\section{Estimates for the problem with shifting environments}
\label{sec:4}

In this section, we consider the solution $u(t,x)$ of the shifting environment problem \eqref{main-eq} in the entire space $(t,x) \in (0,\infty)\times \mathbb R$.

\subsection{A lemma due to Hamel et al.}\label{subsec4-1}

Let $\beta\geq2$ and {$\eta\in\mathbb{R}$} and
$W(t,x)$ be the solution of
\begin{equation}\label{w-eq}
\begin{cases}
W_t = W_{xx} &\text{ for }t>0,~ x> \widehat{X}_{\eta}(t):= \beta t - \eta \log (t+t_0),\\
W(t,\widehat{X}_{\eta}(t)) = 0 &\text{ for }t>0,
\end{cases}
\end{equation}
where the initial data $W(0,x)=W_0(x)\geq(\not\equiv)0$ in $(0,\infty)$ and has
compact support, and $t_0>0$ will be determined later. 
Set $\phi(t,y):= W(t,\widehat{X}_{\eta}(t) + y)$. Then $\phi$ satisfies 
\begin{equation}\label{e.phiphi}
\begin{cases}
\dps\phi_t=\phi_{yy}+\Big(\beta-\frac{\eta}{t+t_0}\Big)\phi_y &\text{ for }t>0,~ y>0,\\
\phi(t,0)=0 &\text{ for }t>0.
\end{cases}
\end{equation}

The following result follows from adapting the arguments due to Hamel et al. \cite[Lemmas 2.1 and 2.2]{hamel2013short}. For completeness, we present the proof of the following result in Appendix \ref{sec.A.1}.

\begin{lemma}
\label{e.a.lem.1}
For each $\beta \geq 2$, $\eta \in\R$ and $t_0 \gg 1$, let $\phi$ be a solution of \eqref{e.phiphi} with nonnegative, compactly supported initial data $\phi_0 \not\equiv 0$. Then 
there exists constants $C,~C_1,~t_1>0$ depending on $\phi_0$ such that
\bea\label{hatphi-formula}
{\phi}(t,y)=\frac{(t+t_0)^{\frac{\beta}{2}\eta-\frac{3}{2}}}{t_0^{\frac{\beta}{2}\eta-1}}y e^{-\frac{\beta}{2}y-\tfrac{\beta^2}4t}C\Big[ e^{-\frac{y^2}{4(t+t_0)}}+h(t,y)\Big],\quad  t>0,\ y\geq 0,
\eea
where
    \begin{align} 
 &|h(t,y)| \leq C_1 \sqrt{\tfrac{t_0}{t + t_0}} &\text{ for }~&t \geq t_1,~ y > 0, \label{e.a.1a}\\
 &\frac{1}{C_1}\leq \frac{{\phi}(t,y)}{t^{-\tfrac32 + \tfrac{\beta\eta}2} y e^{-\tfrac{\beta y}2 - \tfrac{y^2}{4t} - \tfrac{\beta^2}4 t}}\leq C_1  &\text{ for }~&t\geq t_1,~0 \leq y \leq \sqrt{1+t},\label{e.a.2}\\
    &\frac{1}{C_1}{t^{-\tfrac32 + \tfrac{\beta\eta}2}e^{- \tfrac{\beta^2}4t}}\leq {\sup_{y\geq 0} {\phi}(t,y)}\leq C_1 {t^{-\tfrac32 + \tfrac{\beta\eta}2}e^{- \tfrac{\beta^2}4t}}&\text{ for }~&t\geq t_1.\label{e.a.3}
    \end{align}
\end{lemma}

\subsection{Upper and lower estimates}\label{subsec4-2}

{In this subsection, we will estimate the solution 
$u(t,x)$ of \eqref{main-eq}.
We begin by establishing a lower estimate for $u(t,X(t))$.
}

\begin{lemma}\label{l.a.0423.1}
Let $X(t) = \beta t - {\eta} \log (t+1)$ for some 
\begin{equation}\label{e.816.5}
(\beta,\eta) \in (2,\infty) \times \mathbb R \quad \text{ or }\quad (\beta,\eta) \in \{2\} \times (-\infty, \tfrac12) .
\end{equation}
{Let $u(t,x)$ be any solution of \eqref{main-eq} with initial data satisfying \eqref{e.bcc}.}
Then there exist $\delta_1>0$ and $T'>0$ such that 
\begin{equation}\label{e.0816.6}
u(t,X(t)+y) \geq \delta_1  (1+t)^{-\tfrac32 + \tfrac{\beta\eta}2} e^{-\tfrac{\beta}{2}y-(\tfrac{\beta^2}4 -1)t} \quad \text{for }t\geq T', ~ 0 \leq y \leq \sqrt{t}.
\end{equation}
\end{lemma}
\begin{proof}
For $(\beta,\eta)$ as in \eqref{e.816.5}, define
\beaa
\tilde{\phi}(t,x)=
\begin{cases}
e^{t} \phi_{\beta,\eta}(t,x - {X}(t)),&\quad x>X(t),\\
0,&\quad x\leq X(t).
\end{cases}
\eeaa
where $\phi_{\beta,\eta}(t,{x - {X}(t)})$ is the function supported in $[{X(t)},\infty)$ which is given by Lemma \ref{e.a.lem.1}. 
Recalling the definition of generalized sub/supersolutions (see \cite[Definition 4.2]{Berestycki2016shape} or \cite[Definition 1.1.1]{lam2022introduction}), 
then $\tilde\phi(t,x)$ is a generalized subsolution to the linear heat equation 
\[
{\phi_t = \phi_{xx}  +(1-a\chi_{(-\infty,X(t)]})\phi}
\]
 in $\mathbb{R}_+ \times \mathbb{R}$. Moreover, by \eqref{e.816.5}, $\tilde\phi$ satisfies,
 for some $C' >0$ and {$\kappa>0$},
\begin{equation}\label{e.a.423.12}
\sup_{x\in\mathbb{R}}\tilde\phi(t,x) \leq   
C' (1+t)^{-\kappa-1} \quad \text{ for }t\geq 0,
\end{equation}
thanks to \eqref{e.a.3}. 
Now, take $A_0 \in (0,1)$ small enough such that
\bea\label{ic-lem4.2}
u(1,x) \geq A_0\tilde\phi(0,x) \quad \text{ for }x \in \mathbb{R} 
\eea
which is possible since $u(1,x) >0$ for all $x$ and $\phi(0,x)$ is compactly supported. 

Next, we will compare $u(t+1,\cdot)$ with $\underline\phi:= A(t) \tilde\phi(t,x)$, where $A(t)$ 
will be specified later with an ODE. To this end,  we check the differential inequality for $\underline\phi$. Due to the definition of $\tilde{\phi}$, it suffices to verify that 
\beaa
\underline\phi_t - \underline\phi_{xx} - \underline\phi(1-\underline\phi)\leq 0 \quad \text{ for } t>0,\ x\geq X(t). 
\eeaa


 Let $A(t)$ be the solution of the ODE
$$
\frac{d}{dt} A(t) = - C'(1+t)^{-\kappa-1} A(t)^2 \qquad \text{ for } {t>0},\quad \text{ and }\quad A(0) = A_0>0.
$$
Then it is easy to see that
$$
\frac{1}{A(t)} - \frac{1}{A_0} = C' \int_0^t (1+\tau)^{-\kappa-1 }\,d\tau = \frac{C'}{\kappa}\left(1- (1+t)^{-\kappa} \right)
$$
so that $A(t)$ is strictly decreasing and bounded from below by some $A_\infty>0$. Now, we verify the differential inequality for $\underline\phi(t,x):= A(t)\tilde\phi(t,x)$ as follows:
\begin{align*}
\underline\phi_t - \underline\phi_{xx} - 
{ \underline\phi(1- \underline\phi)}
&=  (A \tilde{\phi})_t- (A\tilde{\phi})_{xx} -  A \tilde{\phi} + (A\tilde{\phi})^{2}
\\
&\leq A_t \tilde \phi  +  A^2 \tilde{\phi}^2 \\
&= \tilde\phi A^2 \left[-C'(1+t)^{-\kappa-1} + \tilde\phi \right]\leq 0\end{align*}
where the second inequality is due to $\tilde\phi_t\leq  \tilde\phi_{xx} + \tilde\phi$ in the generalized sense; the following equality is due to the ODE satisfied by $A(t)$ and the last inequality due to \eqref{e.a.423.12}.

Combining  with \eqref{ic-lem4.2}, we can apply the comparison principle to conclude that
$$
u(t+1,x)\geq \underline{\phi}(t,x)\qquad \text{ for }t\geq 0,~ x\in\R.$$ In particular, thanks to \eqref{e.a.2}, there exist $\delta_0>0$ and $T_1\gg 1$ such that 
\begin{align}
\tilde{u}(t+1,X(t)+y)&\geq A(t)\tilde{\phi}(t,X(t)+y)\notag\\&\geq A_{\infty}\tilde{\phi}(t,X(t)+y)\notag \\
&\geq 
\delta_0 y t^{-\frac{3}{2}+\frac{\beta}{2}\eta} e^{-\tfrac{\beta}2 y 
-(\frac{b^2}{4}-1)t} ,\quad t\geq T_1,~ 0\leq y \leq \sqrt t.\label{ineq-tilde-u}
\end{align}
Since $X(t+1)-X(t)=\beta +o(1)$, we can apply 
the parabolic Harnack's inequality to conclude that there exists  $\delta_1>0$ such that 
\beaa 
u(t+1,X(t+1) + y)\geq  \delta_1 t^{-\frac{3}{2}+\frac{\beta}{2}\eta} e^{-\tfrac{\beta}2 y 
-(\frac{b^2}{4}-1)t} ,\quad t\geq T_1+1,~ 0\leq y \leq \sqrt t.
\eeaa
This completes the proof. 
\end{proof}

{
\begin{remark}\label{rmk.0605.1}
    Note that for $X(t) = \beta t -\eta \log (1+t)$ satisfying
\[
(\beta,\eta)  \in (-\infty, 2) \times \mathbb{R} \quad \text{ or }\quad (\beta,\eta) \in \{2\} \times [\tfrac12,\infty),
\]
we can replace the moving boundary $X(t)$ by $Y(t):=2t$ in the proof of Lemma~\ref{l.a.0423.1}, and define
\[
\tilde\phi(t,x):=
\begin{cases}
e^t\phi_{2,0}(t,x-Y(t)),
& x>Y(t),\\
0,
& x\leq Y(t).
\end{cases}
\]
Then the same argument as in the proof of
Lemma~\ref{l.a.0423.1} 
gives $T>0$ and $\delta_1>0$ such that
\[
u(t, 2 t + y) \geq \delta_1 (1+t)^{-\tfrac32} y e^{-y} \quad \text{ for }t\geq T,\ ~ 0\leq y \leq \sqrt t.
\]
\end{remark}
}

\medskip



Next, we provide an upper estimate for  $u(t,X(t))$ by gluing $w$ and $\tilde{\phi}$,
where $w$ is the solution of \eqref{e.0401.2} and $\tilde{\phi}$ is given in the proof of Lemma~\ref{l.a.0423.1}.
The upper estimate does not suffer from the restriction in parameter $\eta$ as in the lower estimate.

\begin{lemma}\label{lem:a.4.2}
Let $u$ be the solution to \eqref{main-eq} with $X(t)$ with $\beta \in [2,\infty)$ and $\eta \in \mathbb{R}$. 
\begin{itemize}
\item[{\rm(a)}] If {$\beta \in [2, 2(\sqrt{a} + \sqrt{1-a})]$,} then {there exist $C>0$ and $T>0$ such that}
\begin{equation}\label{e.816.10}
    u(t,X(t)) \leq C\left[ (x-2\lambda t)^q e^{-\lambda x +(\lambda^2 + 1-a) t} \right]_{x=X(t)}\quad {\text{ for } t\geq T,}
\end{equation}
where
\begin{equation}\label{e.816.3}
    \lambda = \frac{\beta}2 - \sqrt{a} \in  [1-\sqrt{a}, \sqrt{1-a}],\quad q = -\frac32 + \eta \sqrt a.
\end{equation}
\item[{\rm(b)}] If $\beta \in (2(\sqrt a +\sqrt{1-a},\infty)$, then {there exist $C>0$ and $T>0$ such that}
\begin{equation}\label{e.a.4.5}
    u(t,X(t)) \leq C\left[(x-2\lambda_{\min}t)^{-3}  e^{-\lambda_{\min} x +(\lambda_{\min}^2 + 1-a) t} \right]_{x=X(t)} \quad {\text{ for } t\geq T,}
\end{equation}
where $\lambda_{\min} =  \sqrt{1-a}$.
\end{itemize}
    
\end{lemma}
\begin{proof}
Let $X(t) = X_{\beta,\eta}(t)$ be the moving boundary with $\beta \geq 2$ and $\eta \in \mathbb{R}$. 
We first assume $\beta/2 \in [1,\sqrt a + \sqrt{1-a}]$ and prove part (a). 

To this end, set $R=1-a>0$, and $\lambda$, $q$ as in \eqref{e.816.3}. 
Let $\psi(t,x) = \psi_{R,\lambda,q}(t,x)$ be given by the solution to \eqref{e.heat} on the real line and with
 and initial condition \eqref{w-ic-q}. Then by our choice of parameters,  we always have $\beta > 2\lambda$.
This allows the application of
Lemma \ref{lem:Bramson-0813} (i) and (ii) to deduce 
\beaa
\frac{1}{c_1} t^q e^{-\lambda(X(t) - c_\lambda t)} \leq  \psi(t,X(t)) \leq  c_1 t^q e^{-\lambda(X(t) - c_\lambda t)}\quad \text{ for all } t\gg1.
\eeaa
for some $c_1>0$. 
Furthermore, using the  identity  
$\lambda(\beta - c_\lambda) = \tfrac{\beta^2}{4} -1$, we have
\begin{equation}\label{e.a.0423.4.5}
\frac{1}{c_1}t^{-\frac{3}2 + \frac{\beta\eta}2}e^{ - (\frac{\beta^2}{4}-1) t} \leq \psi(t,X(t)) \leq c_1 t^{-\frac{3}2 + \frac{\beta\eta}2 }e^{- (\frac{\beta^2}{4}-1 )t } \quad \text{ for all }t \gg 1,
\end{equation}
Also, 
by Remark \ref{rmk:0423.1b}, 
\begin{equation}\label{e.a.0423.4.6}
\lim_{t\to\infty}    \frac{\psi(t, X(t))}{\psi(t,  X(t)-M +1)} = e^{-\lambda (M-1)} \quad \text{ for each }M \in \mathbb{R}.
\end{equation}
Next, let $\tilde\phi(t,x) = e^{Rt} \phi_{\beta,\eta}(t,x-X(t))$ be given as in the proof of Lemma \ref{l.a.0423.1}. Then by {\eqref{e.a.2}}, there exists $C_1>1$ such that  
\begin{equation*}
\frac{1}{C_1} t^{-\frac{3}{2} + \frac{\beta\eta}2} e^{-(\frac{\beta^2}{4}-1)t}\leq  
\frac{\tilde\phi(t,y + X(t))}{y e^{-\frac{\beta}{2}  y - \frac{y^2}{4t}}} \leq C_1 t^{-\frac{3}{2} + \frac{\beta\eta}2} e^{-(\frac{\beta^2}{4}-1)t}
\end{equation*}
for $0 \leq y \leq \sqrt{1+t}, ~ t \gg 1.$ 
By  \eqref{e.a.0423.4.5} and \eqref{e.a.0423.4.6}, it follows that {for some $C_2>0$,}
\begin{equation}\label{e.0423.7c}
\frac{e^{-\lambda(M-1)}}{C_2}{\psi}(t,X(t)-M+1) <  
\frac{\tilde\phi(t,y + X(t))}{y e^{-\frac{\beta}{2} y - \frac{y^2}{4t}}} < C_2 {\psi}(t,X(t)) 
\end{equation}
for $0 \leq y \leq \sqrt{1+t}, ~ t \gg 1.$ 
Hence, setting $y=M$, we have 
\begin{equation}\label{e.a.0423.7}
\frac{\tilde\phi(t,M + X(t))}{e^{-(\beta -\sqrt a) M/2}} < C_2  M e^{-\sqrt{a} M/2} {\psi}(t,X(t)) \qquad \text{ for } t \gg 1.
\end{equation}
{Note that the constant $C_2$ is {the same as in \eqref{e.0423.7c}}, and is independent of $M$.} So we can choose $M \gg 1$ such that
\begin{equation}\label{e.MM}
  C_2 M e^{-\sqrt{a} M} <1 \quad \text{ and } \quad   \frac{e^{-\beta/2}}{e^{-(\beta - \sqrt{a})M/2}} \cdot \frac{e^{-\lambda (M-1)}}{C_2} =\frac{e^{-\beta/2+\lambda }}{C_2}  e^{\sqrt{a} M/2}  >2 
\end{equation}

We claim that with the above choice of $M$, there exists $t_1 >1$ such that 
\begin{equation}\label{e.c.1}
    \frac{\tilde\phi(t,M + X(t))}{e^{-(\beta -\sqrt a) M/2}} < {\psi}(t,X(t)) \quad \text{ for }t \geq t_1,
\end{equation}
and
\begin{equation}\label{e.c.2}
    \frac{\tilde\phi(t,1 + X(t))}{e^{-(\beta -\sqrt a) M/2}} > {\psi}(t,X(t)- M+1) \quad \text{ for }t \geq t_1. 
\end{equation}
Indeed, \eqref{e.c.1} follows from \eqref{e.a.0423.7} and \eqref{e.MM}. To see \eqref{e.c.2}, set $y=1$ in \eqref{e.0423.7c} to get
\begin{align}
\frac{\tilde\phi(t,1 + X(t))}{e^{-(\beta - \sqrt{a})M/2}}  &> \frac{e^{-\beta/2}}{e^{-(\beta - \sqrt{a})M/2}} \frac{e^{-\lambda(M-1)}}{C_2}{\psi}(t,X(t)-M+1)\notag\\
&>2 {\psi}(t,X(t)-M+1)\quad {\text{ for all } t\gg1},  \label{e.c.3}
\end{align}
by the second part of \eqref{e.MM}.
This proves \eqref{e.c.2} provided we choose $t_1$ large enough.

We now construct the supersolution by gluing $w$  and a constant multiple of $\tilde\phi(t,x+M)$. 
\begin{equation}
    \overline{u}_1(t,x) := \begin{cases}
    {\psi}(t,x) &\text{ for } x < X(t)-M+1,\\
    \min\left\{\psi(t,x),  e^{(\beta - \sqrt{a})M/2}{\tilde\phi(t,x+ M)}\right\} &\text{ for }X(t)-M+1 \leq  x < X(t),\\
    e^{(\beta - \sqrt{a})M/2}{\tilde\phi(t,x+ M)} &\text{ for }x \geq X(t).
    \end{cases}
\end{equation}
We claim that 
the function $\overline{u}_1$ is a generalized supersolution of \eqref{main-eq} in the domain $[t_1,\infty)\times \mathbb R$. Indeed, since ${\psi}(t,x)$ is a classical supersolution for $x < X(t)$, and the same is true for $\overline\phi(t,x):= 
{e^{(\beta - \sqrt{a})M/2}}\tilde\phi(t,x+M)$ 
for $x> X(t) - M$.
Note that  \eqref{e.c.1} and \eqref{e.c.2} together imply that, in a neighhorbood of $x=X(t)$ and $x=X(t)-M+1$, $\overline{u}_1$ is twice differentiable and is a classical supersolution, thus $\overline{u}_1$ satisfies the gluing criterion for generalized supersolution \cite[Remark 1.1.2]{lam2022introduction}. 

Next, choose $M_1>1$ such that $X(t + t_1) \leq X(t) + M_1$ for all $t\geq 0$, and define
$$
\bar{u}_2(t,x) = \bar{u}_1(t+t_1,x+M_1) \quad \text{ for }t\geq 0,~x\in\mathbb R.
$$
We claim that $\bar{u}_2$ is a generalized supersolution of \eqref{main-eq} in the entire domain $[0,\infty)\times \mathbb{R}$. Indeed, write $r(t,x) = 1-a\chi_{\{x \leq X(t)\}}$, then the following holds in the generalized sense: 
\begin{align*}
(\bar{u}_2)_t - (\bar{u}_2)_{xx} &= \bar{u}_2 \left(r(t+t_1, x + M_1) - \bar{u}_2 \right)\\
 &=    \bar{u}_2 \left(1- a\chi_{\{x+ M_1 \leq X(t+ t_1)\}} - \bar{u}_2 \right)\\
  &\geq    \bar{u}_2 \left(1- a\chi_{\{x \leq X(t)\}} - \bar{u}_2 \right)
\end{align*}
where the first equality is due to the supersolution property of $\bar{u}_1$, and the inequality is due to the choice of $M_1$ which implies
$$
\{(t,x):~x + M_1 \leq X(t+t_1)\}  \subseteq \{(t,x):~x \leq X(t)\}.
$$

Now, since initial data $u_0(x)$ has compact support, we may assume without loss that 
$$
u_0(x) \leq \bar{u}_2(0,x) \quad \text{ in }\mathbb{R},
$$
since $K\bar{u}_2$ remains a generalized supersolution of \eqref{main-eq} for each constant $K>1$. It then follows from the comparison principle for generalized super/subsolutions that
$$
u(t,x) \leq \bar{u}_2(t,x) \quad \text{ for }t>0,~x\in \mathbb R.
$$
Particularly,  we have ($C$ may change from line to line)
\begin{align*}
u(t+1,X(t+t_1+1) - M_1) &\leq \bar{u}_2(t+1, X(t+t_1+1) - M_1)\\
&= \bar{u}_1(t+t_1+1, X(t+t_1+1))\\
&\leq {\psi}(t+t_1+1, X(t+t_1+1))\\
&\leq C {t^{q+\lambda \eta }} e^{(-\beta \lambda  + \lambda^2 + 1-a)t }\\  
&\leq 
C\left[(x-2\lambda t)^q e^{-\lambda x + (\lambda^2 + 1-a)t}\right]_{x=X(t)}
\end{align*}
for $t\gg 1$,  
where we used 
\eqref{e.a.0423.4.5}
in the 
{third} inequality. 
Finally, since 
\[
|X(t) - X(t+t_1+1)| \to(t_1 +1) \beta \qquad {\text{ as } t\to\infty}.
\]
Since it remains bounded, it follows by Harnack's inequality that
$$
u(t,X(t)) \leq C u(t+1,X(t+t_1+1)-M_1) \leq C\left[(x-2\lambda t)^q e^{-\lambda x + (\lambda^2 + 1-a)t}\right]_{x=X(t)}
$$
for $t \gg 1$. This proves assertion (a).

For assertion (b), let $\beta > 2(\sqrt{a} + \sqrt{1-a})$ be given. 
Define 
$$
\bar{u}_3(t,x) = M_3e^{-\frac{\beta}2 x + (\frac{\beta^2}4 + 1)t},
$$
where $M_3$ is chosen large enough so that $\bar{u}_3(0,x) \geq u(0,x)$, then it is easy to see that $\bar{u}_3(t,x)$ is a classical supersolution of \eqref{main-eq}, and that by the comparison principle, $u \leq \bar{u}_3$ for all $t>0$, $x\in \mathbb{R}$. In particular,
\begin{equation}\label{e.c.8a}
u(t,X(t)) \leq \left[M_2 e^{-\frac{\beta}2x+ (\frac{\beta^2}{4} +1)t}\right]_{x=X(t)} \qquad \text{ for }t\geq 0. 
\end{equation}
Observe now for $t\gg 1$,
\begin{align}
\left[ \frac{e^{-\frac{\beta}2x+ (\frac{\beta^2}{4} +1)t} }{(x-2\sqrt{1-a}t)^{-3}e^{-\sqrt{1-a}x + 2(1-a)t}} 
 \right]_{x=X(t)} 
 &\leq C t^3
\left[ {e^{-(\frac{\beta}2-\sqrt{1-a})x+ (\frac{\beta^2}{4} +2a-1)t} }
 \right]_{x=X(t)} \noindent \notag \\
 &\leq  t^\gamma e^{-[ (\tfrac{\beta}2  - \sqrt{1-a})^2 - a] t} \quad \text{ for some }\gamma \in \mathbb R.
 \label{e.c.8b}
\end{align}
Using $\beta/2 > \sqrt{a} + \sqrt{1-a}$, we see that the last expression in \eqref{e.c.8b} tends to zero exponentially as $t\to\infty$.
%
It follows from \eqref{e.c.8a} and \eqref{e.c.8b} that \eqref{e.a.4.5} holds. 
\end{proof}
\begin{remark}\label{rk:818.1}
Under the hypothesis of Lemma \ref{lem:a.4.2}(a) (i.e. in the case $\lambda \in (0,{\sqrt{1-a}})$), we have
$$
u(t,x) \leq \overline{u}_2(t,x) \leq \psi(t+t_1,x+M_1)  \quad \text{ for }t>0,~x \leq X(t).
$$
\end{remark}

\subsection{Proof of Theorems \ref{thm:standard}, 
\ref{thm:pulling},  \ref{thm:cpulling}, and \ref{thm:no-pull}}

\begin{proof}[Proof of Theorem \ref{thm:pulling} for the case $\beta \in (2,2\sqrt{a} + 2\sqrt{1-a})$]

Let $\lambda=\frac{\beta}{2}-\sqrt{a}$ and $c_{\lambda}=\lambda+(1-a)/\lambda$. Then 
\begin{align}\label{ID-1}
\lambda \beta-\lambda^2-(1-a) = (\frac{\beta}{2}-\sqrt a) \beta- (\frac{\beta}{2}-\sqrt a)^2 -(1-a)=
\frac{\beta^2}{4}-1.
\end{align}
In view of Lemma~\ref{l.a.0423.1} and \eqref{ID-1}, there exists $\delta_1>0$ and $T>0$ such that for all $t\geq T$,
\begin{equation}\label{thm1.5-lowerbd}
u(t,X(t)) \geq \delta_1(1+t)^{-\tfrac32 + \tfrac{\beta\eta}2} e^{-(\tfrac{\beta^2}4 -1)t}=\delta_1(1+t)^{-\tfrac32 + \tfrac{\beta\eta}2} e^{(-\lambda \beta+\lambda^2+1-a)t}
\end{equation}
Therefore, by \eqref{thm1.5-lowerbd}, we have 
\begin{equation}\label{e.816.11}
\liminf_{t \to +\infty} \left[(x - 2\lambda t)^{-q} e^{\lambda x - (\lambda^2 +1-a)t} u(t,x)\right]_{x=X(t)} \geq  k_0 \qquad \text{ for some }k_0>0,
\end{equation}
where
\beaa
q=-\frac{3}{2}+(\frac{\beta}{2}-\lambda)\eta=-\frac{3}{2}+\sqrt{a}\eta.
\eeaa
Setting $R=1-a$, $\lambda \in (0,R)$, we may apply Lemma~\ref{lem:liminf-case}(i) and Remark~\ref{rk:3.2} to get
\begin{equation}\label{lower-est-thm1.5}
\liminf_{t\to\infty }\Big[\inf_{ -\infty< x \leq X_{\eta}(t)}\Big(u(t,x )  -\Phi_{\lambda,1-a}(x- m_{\lambda,q}(t)-\frac{1}{\lambda}\log k_0)\Big)\Big] \geq  0,
\end{equation}
where $m_{\lambda,q}$ is defined in \eqref{m^w}.

Next, 
we combine  
\eqref{e.816.10} of
Lemma~\ref{lem:a.4.2},  Lemma~\ref{lem:limsup-case}(i) and Remark~\ref{rk:3.1} to get
\begin{equation}\label{upper-est-thm1.5}
\limsup_{t\to\infty }\Big[\sup_{ -\infty<x \leq X_{\eta}(t)}\Big(u(t,x )  -\Phi_{\lambda,1-a}(x- m_{\lambda,q}(t)-h_0)\Big)\Big] \leq  0 \quad \text{ for some }h_0 \in \R.
\end{equation}

Combining \eqref{lower-est-thm1.5} and \eqref{upper-est-thm1.5}, we obtain \eqref{level-set}. Moreover, {similar to the proof of {Proposition~\ref{thm1}},}
we can apply the Liouville type result \cite[Theorem 3.5]{beresycki2007generalized} to conclude \eqref{thm1.5-cov}. This completes the proof.
\end{proof}

\begin{proof}[Proof of Theorem \ref{thm:pulling} for the case $\beta=2$ and $\eta\in(-\infty,1/2)$] 
{Note that the proof of Theorem~\ref{thm:pulling} for $\beta\in(2,2(\sqrt{a}+\sqrt{1-a}))$ relies on 
$\zeta^{-1}(x):=X(t)= \beta t - \eta \log(1+t)$ satisfying \eqref{zeta-cond}. However, this condition does not hold for $\beta=2$ and thus a different proof is needed.}

{{Let $X(t) = \beta t - \eta \log(1+t)$ where  $\beta=2$ and $\eta < \tfrac12$. Then set
\beaa
\lambda = 1-\sqrt{a} \quad \text{ and }\quad
q=-\frac{3}{2}+\sqrt{a}\eta.
\eeaa
Let $\xi_{b}$ be defined in \eqref{location-u}. 
We first deal with the upper bound of $\xi_b(t)$. By Remark \ref{rk:818.1}, there exists  constants  $t_1,M_1$ such that 
$$
u(t,x) \leq \psi(x + M_1, t+ t_1)\quad \text{ for all }t \geq 0,~x \leq X(t).
$$
Fix any {$\delta\in(0,\sqrt{a})$} and then $\ep>0$ small enough so that
$$
[2\lambda + \delta, \infty) \supset [2-\delta, 2+\delta]\quad \text{ and }\quad \ep \delta < \frac{\delta^2}{19}.
$$
It follows from the estimate \eqref{e.0813.1a},that  
\begin{align}
u(t,x) &\leq C  \left[t^q e^{-\lambda(x-2t)} + e^{-(\lambda-\ep)(x-2t) - \frac{\delta^2}{19}t} \right]\notag \\   
&\leq  C  t^q e^{-\lambda(x-2t)} (1- e^{-\delta' t}) \notag \\
&\leq C e^{-\lambda( x - 2t -\frac{q}{\lambda} \log t)}(1- e^{-\delta' t})
\label{e.819.1}
\end{align}
for $t\gg 1$ and ${(2-\delta)t<} x < X(t) = 2t - \eta\log(1+t)$. This implies
\begin{equation}\label{upper-beta2}
    \xi_b(t) \leq 2t + \frac{q}{\lambda} \log t + O(1) \quad \text{ for }t\gg 1.
\end{equation}

}

It remains to show the lower bound. 
To this end, apply Lemma \ref{l.a.0423.1} to get
\begin{equation}\label{u-low0821}
    u(t,X(t)) \geq \delta_1 (1+t)^{-\tfrac32 + \eta} \quad \text{ for }t \geq 1. 
\end{equation}
for some $\delta_1>0$.
Define 
\begin{equation}
    \underline{u}(t,x):= {\kappa}\Phi_{\lambda,1-a}\left(x- 2t -\tfrac{q}{\lambda} \log t + x_0\right).
\end{equation}
where $\kappa\in(0,1)$ is an arbitrary constant and 
$\Phi_{\lambda,1-a}$
is as in the statement of {Theorem \ref{thm:pulling}}
and $x_0 >0$ is to be determined later. It is clear that $\underline{u}(t,x)$ 
satisfies 
\bea\label{Dineq0821}
\underline{u}_t-\underline{u}_{xx}-\underline{u}(1-a-\underline{u})\leq 0,\quad t>0,\  x\in\R. 
\eea

Let $z(t) = X(t) - 2t - \tfrac{q}{\lambda} \log t + x_0$. Then the condition $\eta <\tfrac12$ implies that $z(t) \to \infty$ as $t\to\infty$. 
Thanks to the asymptotic property \eqref{TW-AS}, we have
\begin{align*}
t^{\tfrac32 - \eta}e^{\lambda x_0}  \underline{u}(t,X(t))&=
\exp\Big(\lambda(X(t) - 2t  - \tfrac{q}{\lambda} \log t + x_0)\Big)\underline{u}(t,X(t))\\
&= e^{\lambda z(t)}{\kappa} \Phi_{\lambda,1-a}(z(t)) \to {\kappa} \qquad \text{ as }t\to\infty.
\end{align*}
Hence, we can choose $x_0$ large enough and $T>0$ to ensure that 
\begin{equation}\label{BC0821}
    \underline{u}(t,X(t)) \leq \delta_1 (1+t)^{-\tfrac32 + \eta } \leq u(t,X(t))\quad \text{ for } t \geq {T},
\end{equation}
where we used \eqref{u-low0821}.

On the other hand, since $\liminf_{x\to-\infty} u_0(x)>0$, we have  $\liminf_{t\to\infty}\inf_{x\leq0}u(t,x)\geq 1-a$.
Therefore, for such fixed $\kappa\in(0,1)$, there exists $T'>T$ such that
\beaa
u(t,x)>\kappa(1-a)=\kappa \Phi_{\lambda,1-a}(-\infty)>  \underline{u}(t, x)\quad \text{ for } t\geq T',\ x\leq 0,
\eeaa
where the last inequality holds since $\Phi_{\lambda,1-a}'<0$. Moreover, since $u(T',x)>0$ for all $x\in[0, X(T')]$, if necessary we may choose $x_0$ larger such that
\begin{equation}\label{ic0821}
    \underline{u}(T',x) \leq u(T',x) \quad \text{ for all } x \leq {X(T')},
\end{equation}
}

Combining \eqref{Dineq0821}, \eqref{BC0821} and \eqref{ic0821},
$\underline{u}(t,x)$ is a classical subsolution in the domain 
$$
\Omega_3:=\{(t,x):~t>{T'},~ x \leq X(t)\}.
$$
It follows by comparison that $u(t,x) \geq \underline{u}(t,x)$ in $\Omega_3$. This proves the lower bound 
\begin{equation}\label{lower-beta2}
    \xi_b(t) \geq  2t + \frac{q}{\lambda} \log t + O(1) \quad \text{ for }t\gg 1.
\end{equation}
Combining \eqref{upper-beta2} and \eqref{lower-beta2}, we obtain \eqref{level-set}. 

Next observe that
the upper bound \eqref{e.819.1} and the lower bound $u(t,x) \geq \underline u(t,x)$  hold in $\Omega_3$ and in particular in 
$$
\{(t,x):~ t \geq \bar T,~ \xi_b(t) - \bar\delta t < x < \xi_b(t) + \bar\delta t\} \quad \text{ for some } \bar T, \bar\delta >0.
$$
The assertion \eqref{thm1.5-cov}, i.e. the convergence to traveling wave profile,  
follows from the arguments in \cite{hamel2013short} together with the procedure employed in the proof of {Proposition~\ref{thm1}}. This completes the proof.
\end{proof}


\begin{proof}[Proof of Theorem \ref{thm:cpulling}]
Let 
$\beta = 2(\sqrt{a} + \sqrt{1-a})$, $\lambda = \sqrt{1-a}$, then $\beta > 2\lambda$ and we can repeat the proof of Theorem~\ref{thm:pulling}, except to replace $m_{\lambda,q}(t)$ by $\tilde{m}_q(t)$ in \eqref{lower-est-thm1.5} and \eqref{upper-est-thm1.5}. We omit the detailed proof here.
\end{proof}

\begin{proof}[Proof of Theorem~\ref{thm:no-pull}] 
Let $\xi_{b}$ be defined in \eqref{location-u}. 
Since we have $u_t-u_{xx}-(1-a-u)u\geq 0$ for $x\in \R$ and $t>0$.
By comparison with the solution of 
$u_t-u_{xx}-(1-a-u)u=0$ with the same initial data as in \eqref{main-eq}, from \cite[Theorem 1.2]{hamel2013short} we see that there exist $T_1>0$ and $C_1>0$ such that for any ${b} \in (0,1-a)$, 
\[
\xi_{b} (t)\geq  2\sqrt{1-a} t-\frac{3}{2\sqrt{1-a}}\log t+C_1,\quad t\geq T_1.
\]
and
\begin{equation}\label{lower-bd-1.5}
\liminf_{t\to\infty} \inf_{x \in \mathbb R}
\left[u(t,x) - \Phi\Big(x- 2\sqrt{1-a} t  + \frac{3}{2\sqrt{1-a}} \log t + C_1\Big)\right] \geq 0.
\end{equation}

We next deal with the upper bound of $\xi_{b} (t)$.
By Lemma \ref{lem:a.4.2}(b), we have
$$
u(t,X(t)) \leq C\left[ (x-2\lambda_{\min} {t})^{-3} e^{-{\lambda_{\min} X(t)}+(\lambda^2_{\min} + 1-a)t} \right] \qquad \text{ for }t\gg 1.
$$
By taking $R = 1-a$, $\lambda = \sqrt{1-a}$, {$q=-3$} and $\zeta^{-1}(t) = X(t)$,  
one can apply Lemma \ref{lem:limsup-case}(ii) {and Remark~\ref{rk:3.1}} to deduce that {for some $C_2>0$},
\begin{equation}\label{e.0607.1}
\limsup_{t\to\infty} \sup_{x \leq X(t)}\left[ u(t,x) - \Phi(x- 2\sqrt{1-a} t  + \frac{3}{2\sqrt{1-a}} \log t - C_2)\right]\leq 0
\end{equation}
and thus 
\[
    \xi_{b}(t) \leq 2\sqrt{1-a} t  - \frac{3}{2\sqrt{1-a}} \log t + C_2 \qquad \text{ for }t\gg 1.
\]
Finally, \eqref{thm1.7-conv} can be obtained in a similar manner from the estimates \eqref{lower-bd-1.5} and \eqref{e.0607.1}, as in the proof of {Proposition~\ref{thm1}}, via the Liouville type result \cite[Theorem 3.5]{beresycki2007generalized}. This completes the proof.
\end{proof}

{
Next, we prove Theorem~\ref {thm:standard}.
\begin{proof}[Proof of Theorem \ref{thm:standard}]
We only need to prove the lower estimate, as the upper estimate directly follows from the classical result of Bramson since any solution to \eqref{main-eq} is a subsolution of the homogeneous KPP equation.

By Remark \ref{rmk.0605.1}, 
there exist $T_1>0$ and $\delta_1>0$ such that
\begin{equation}\label{e.0605.2}
   u(t, 2t + \sqrt{t}) \geq \delta_1 \sqrt t e^{-\sqrt t} t^{-\tfrac32} \quad \text{ for }t \geq T_1. 
\end{equation}
Next, let $\Phi_{\min, 1}(s)$ and  $\tilde\Phi(s)$ be the traveling wave solution of $2\Phi'+\Phi''+\Phi(A-\Phi)=0$ connecting the equilibrium states $A$ and $0$, with $A =1$ and $1-a$, respectively. Then there exist $B>0$ and $\tilde B>0$ such that
\[
\Phi_{\min, 1}(s) \sim B s e^{-s} \quad \text{ and }\quad \tilde\Phi(s) \sim \tilde B e^{-(1-\sqrt a)s} \quad \text{ as }s \to +\infty,
\]
where we can further normalize these functions such that $\Phi_{\min, 1}(0) = \tfrac12$ and $\tilde\Phi(0) = \tfrac12 (1-a)$. Since $\tilde\Phi(-\infty) = 1-a$ and $\tilde\Phi(+\infty) = 0$, we can choose $\ell>0$ such that
\[
\inf_{x \leq -\ell}\tilde\Phi(x) \geq  \frac{2(1-a)}{3}  \qquad \text{ and }\qquad \sup_{x \geq \ell}\tilde\Phi(x) \leq  \frac{1-a}3 .
\]
We define the following subsolution $\underline u(t,x) = V(x-2t + \tfrac32 \log t)$, where
\begin{equation}
   V(y):= \begin{cases}
  \tfrac{1-a}2 \Phi_{\min, 1}(y - x_1)     &\text{ for }y > 2\ell,\\
  \max\{\tfrac{1-a}2 \Phi_{\min, 1}(y - x_1) , \tilde\Phi(y - \ell)\}  &\text{ for }0\leq y \leq 2\ell,\\
  \tilde\Phi(y - \ell)&\text{ for }y < 0,
   \end{cases} 
\end{equation}
for some constant $x_1>0$.
We can ensure that $V(y)$ is continuous, thanks to the choice of $\ell$ above and by choosing $x_1>0$ large enough so that $\tfrac{1-a}2 \Phi_{\min, 1}(y - x_1) \in [\tfrac{1-a}{3},\tfrac{2(1-a)}{3}]$ for $y \in (-\infty,2\ell]$. Since $X(t) \leq 2t -\tfrac32 \log t$, we can check  that for any $\delta \in (0,1)$, $\delta \underline{u}(t,x)$ is a subsolution of \eqref{main-eq}. 

Using the asymptotic properties of $\Phi_{\min,1}$, we have
\begin{align*}
\underline{u}(t,2t  + \sqrt t) \sim B\tfrac{1-a}{2}(\sqrt t +\tfrac32 \log t -x_1) e^{-(\sqrt t +\tfrac32 \log t -x_1)} \quad \text{ as } t\to\infty,
\end{align*}
which implies that for some $B'>0$,
\begin{align*}
\underline{u}(t,2t  + \sqrt t) \sim B' \sqrt te^{-\sqrt t}  t^{-3/2} \quad \text{ as } t\to\infty.
\end{align*}
Together with \eqref{e.0605.2}, we can then choose $0<\delta \ll 1 $ such that 
\[
\delta \underline{u}(t,2t+\sqrt{t}) \leq u(t,2t+\sqrt{t})\quad \text{ for }   t\geq T_1.
\]
Moreover, since $u(T_1,x)>0$ for all $x\in\mathbb R$,
$\liminf_{x\to-\infty}u(T_1,x)>0$, and
$\underline u(T_1,\cdot)$ is bounded, by decreasing $\delta>0$ if necessary, we may also ensure that
\[
\delta\underline u(T_1,x)
\leq
u(T_1,x)
\qquad\text{for all }x\leq 2T_1+\sqrt{T_1}.
\]
Consequently,  we can choose $0<\delta \ll 1 $ such that $\delta \underline{u}(t,x) \leq u(t,x)$ 
in the set
\[
\{(t,x):~ t \geq T_1,~ x = 2t + \sqrt t\} \cup \{(t,x):~ t = T_1,~ x \leq 2T_1 + \sqrt{T_1}\}.
\]
Hence, the maximum principle yields $u(t,x) \geq \delta \underline{u}(t,x)$ for $t \geq T_1$ and $x \leq 2t + \sqrt t$. This implies that
\[
\xi_b(t) \geq 2t - \tfrac32 \log t  - C \quad \text{ for }t \gg 1.
\]
The rest of the theorem follows as before.
\end{proof}
}

\section*{Acknowledgements}

The authors dedicate this paper in memory of Prof. Ka-Shing Lau, who was professor and chairman of the Department of Mathematics at the Chinese University of Hong Kong. 
KYL is supported partly by the National Science Foundation grant DMS-2325195. 
CHW is supported partly by the National Science and Technology Council of Taiwan (NSTC 113-2628-M-A49-004-MY4, 113-2918-I-A49-002) and the National Center for Theoretical Sciences. CHW also thanks the Mathematical Research Institute (MRI) at the Ohio State University for the hospitality during his sabbatical
visit in 2024-2025. This work was partially completed when the authors attended the Workshop on ``Dynamic Models in Biology" (24w5209) at the Banff International Research Station (BIRS), and the workshop on ``Mathematical Modeling of Biological Interfacial Phenomena" at the
Institute for Mathematical and Statistical Innovation (IMSI) (supported by National Science Foundation Grant No. DMS-1929348). The authors thank BIRS and IMSI for the support and stimulating research environments.

\appendix

\section{Appendix: Proof of Lemma~\ref{e.a.lem.1}}\label{sec.A.1}

\begin{proof}[Proof of Lemma~\ref{e.a.lem.1}]
Consider 
\begin{equation*}
\begin{cases}
\dps\phi_t=\phi_{yy}+\Big(\beta-\frac{\eta}{t+t_0}\Big)\phi_y &\text{ for }t>0,~ y>0,\\
\phi(0,y) = \phi_0(y)\quad \text{ for }y>0,\quad \qquad &\phi(t,0)=0 \quad \text{ for }t>0,
\end{cases}
\end{equation*}
where the initial data $\phi_0$ is nonnegative, nontrivial, and 
compactly supported in $(0,\infty)$.

Introduce, as in the proof of \cite[(20) in Lemma 2.1]{hamel2013short},  
the self-similar variables
\beaa
\tau:=\ln\Big(\frac{t+t_0}{t_0}\Big), \quad z:=\frac{y}{\sqrt{t+t_0}}
\eeaa
and define $v(\tau,z)$ by
\bea\label{v-phi}
v(\tau,z)=
\exp\Big\{-(\frac{\beta}{2}\eta-1) 
\ln(\frac{t+t_0}{t_0}) \Big\}e^{\frac{\beta^2}{4} t+\frac{\beta}{2}y}\phi(t,y).
\eea
Then one can verify by direct calculation that $v(\tau,z)$ satisfies\footnote{This is equation (22) in \cite{hamel2013short}. Note that the variables $(\tau,y)$ in \cite{hamel2013short} correspond to  $(\tau,z)$ here.} 
\bea\label{v-tau-z-sys}
\begin{cases}
v_{\tau}-v_{zz}-\frac{z}{2}v_z-v=-\ep e^{-\tau/2}v_z,&\quad t>0, \ z>0,\\
v(\tau,0)=0,&\quad t>0.
\end{cases}
\eea
where $\ep:=\eta /\sqrt{t_0}$ \footnote{ Note that the parameters $(\lambda^*,r)$ in \cite{hamel2013short} corresponds respectively to  $(\frac{\beta}{2},\frac{\beta}{2}\eta)$ here.} (Here we choose $t_0\gg1$ such that $\ep<1$.).
By \cite[Lemma 2.2]{hamel2013short},
there exists $\ep_0>0$ such that, for all compact sets $K$ of $[0,\infty)$, there is $C_K>0$ such that for any $\ep\in(0,\ep_0)$,
\bea\label{v-from-hamel}
v(\tau,z)=z\Big(\frac{e^{-z^2/4}}{2\sqrt{\pi}}(\int_0^{\infty}\xi v_0(\xi)d\xi+O(\ep))+e^{-\tau/2}\tilde{v}(\tau,z)\Big),\ \tau>0,\ z>0.
\eea
It is also proved that for each compact interval $K \subset [0,\infty)$, there exists $C_K$ such that
\bea\label{tilde-v-est}
|\tilde{v}(\tau,z)|\leq C_Ke^{-z^2/8} \qquad\text{ for all } 
 \tau \geq 1, \ z\in K.
\eea

For our purpose, we shall explain in the following that $C_K$ can in fact be chosen independent of $K$. More precisely, we claim that there exists $\hat{C}>0$ such that
\bea\label{v-tilde-est2}
 | \tilde{v}(\tau,z) |\leq \frac{\hat{C}}{1+z}e^{-z^2/8} \qquad \text{ for all } \tau\geq 1, \ z\geq 0.
\eea
To this end, as in \cite[Lemma 2.2]{hamel2013short}, let us set
\bea \label{e.vwvw}
v(\tau,z):= e^{-z^2/8} w(\tau,z),
\eea
and
the function ${w}(\tau,z)$ satisfies the equation (see \cite[eqn. (23)]{hamel2013short}) 
\bea
{w}_\tau+Mw = -\ep e^{-\tau/2}(w_z-\frac{z}{4}w),\quad  \tau>0, \ z>0,\quad w(\tau,0)=0.
\eea
where $M$ is a self-adjoint operator defined by 
\beaa
M:=-\partial_{zz}+(\frac{z^2}{16}-\frac{3}{4}).
\eeaa


Define $\tilde{w}$ by projection in the $z$ variable 
\bea\label{e.wwww}
w=\langle w, e_0 \rangle_{L^2([0,\infty))} e_0+\tilde{w},
\eea
where $e_0(z):=ze^{-z^2/8}/{(2\sqrt{\pi}))^{1/2}}$ is the normalized 
principal eigenfunction of $M$.

From \cite[eqn. (27)]{hamel2013short}, $\tilde{w}$ satisfies 
\bea\label{w-tilde-eq}
\tilde{w}_{\tau}+M\tilde{w}
=-\ep e^{-\tau/2} \Big(F(\tau,z)+\tilde{w}_z-\frac{z}{4}\tilde{w}\Big),
\eea
where $F(\tau,z)= \langle (e_0)_z + \tfrac{z}4 e_0, w(\tau,\cdot)\rangle e_0 + \langle e_0, w(\tau,\cdot)\rangle ((e_0)_z- \tfrac{z}4 e_0) $ satisfies
\bea
\label{e.0314.1}
\|F\|_{L^\infty((0,\infty)\times (0,\infty)}\leq C_0
\eea
since $\sup_{\tau > 0} \|w(\tau, \cdot)\|_{L^2} \leq C$ 
for some $C>0$ depending only on $\phi(0,y)$
(thanks to \cite[eqn. (25)]{hamel2013short}).
Furthermore, from \cite[p.282]{hamel2013short} there exists $C_1>0$ depending only on $\phi(0,y)$ such that   
\begin{equation}
\label{e.0310.2}
\|\tilde{w}(\tau,\cdot)\|_{L^2([0,\infty))}\leq C_1 e^{-\tau/2}.    
\end{equation}

We now prove that there exists $\hat{C}>0$ depending only on the initial data such that
 \bea\label{w-tilde-L-infty}
 \sup_{z \geq 1}|\tilde w(\tau,z)| \leq \frac{1}{2}\hat{C} e^{-\tau/2} \quad \text{ for all } \tau \geq 1.
 \eea
From \eqref{w-tilde-eq} and 
  using $\tfrac{z^2}{16} + \ep \tfrac{z}{4} \geq -\frac{\ep^2}{4}$, the function $\tilde{w}$ satisfies
\begin{equation}
    \tilde{w}_\tau - \tilde{w}_{zz} + \ep e^{-\tau/2} \tilde{w}_z -  \frac{\ep^2 + 3}{4} \tilde w \leq e^{-\tau/2}F(\tau,z) \quad \text{ for }\tau>0,~ z>0.
\end{equation}
By the local maximum principle (e.g. \cite[Theorem 7.36]{lieberman1996second}), there exists $C_2>1$ independent of $\tau\geq 1$ and  $z_0 \geq 1$ such that
\bea
 &&\|\tilde{w}(\tau,\cdot)\|_{L^\infty ( [ z_0-1/2, z_0+1/2])}\notag \\
 &\leq& C_2 \Big(\|\tilde{w}\|_{L^1 ( (z_0-1, z_0+1)\times (\tau-1,\tau])} 
 + \|e^{-\tau/2}F\|_{L^\infty( [ (z_0-1, z_0+1)\times (\tau-1,\tau])}\Big) \notag\\
  &\leq& C_2 e^{-\tau/2}\Big(e^{\tau/2}\|\tilde{w}\|_{L^1 ( (z_0-1, z_0+1)\times (\tau-1,\tau])} 
 + e^{1/2}\|F\|_{L^\infty( [ (z_0-1, z_0+1)\times (\tau-1,\tau])}\Big) \label{e.0310.3}
\eea
Next, we apply Cauchy-Schwarz inequality and \eqref{e.0310.2}  to get
 \begin{align*}
 \|\tilde{w}\|_{L^1 ( [ z_0-1, z_0+1]\times (\tau-1,\tau])}&\leq  \int_{\tau-1}^\tau \sqrt{2}\|\tilde{w}(\tau',\cdot)\|_{L^2}\,d\tau' \leq \sqrt{2}C_1 e^{-(\tau-1)/2} = C_3 e^{-\tau/2},
\end{align*}
where $C_3=2C_1 e^{-1/2}$ is independent of $\tau\geq 1$ and  $z_0 \geq 1$. Substituting this and \eqref{e.0314.1} into \eqref{e.0310.3}, we have proved that for some $\hat{C}>0$, \eqref{w-tilde-L-infty} holds. Finally, using the relation between $\tilde{v}$ and $\tilde{w}$ contained in \eqref{v-from-hamel}, \eqref{e.vwvw} and \eqref{e.wwww} 
(see also \cite[p.282]{hamel2013short}),
\beaa
 \tilde{v}(\tau,z)=\frac{\tilde{w}(\tau, z)e^{-z^2/8+\tau/2}}{z},
\eeaa
we deduce from \eqref{w-tilde-L-infty} that 
\beaa
 | \tilde{v}(\tau,z) |\leq \frac{\hat{C}}{2z}e^{-z^2/8}\leq \frac{\hat{C}}{1+z} e^{-z^2/8} \quad \text{ for all } t\geq 1, \ z\geq 1.
\eeaa
By enlarging $\hat C$ if necessary, we may combine
 with \eqref{tilde-v-est} (with $K=[0,2]$ therein) to ortain \eqref{v-tilde-est2}.
Combining \eqref{v-phi}, \eqref{v-from-hamel}, and the estimate \eqref{v-tilde-est2}, we ortain \eqref{hatphi-formula} with \eqref{e.a.1a}.

We now prove \eqref{e.a.2} and \eqref{e.a.3}. For \eqref{e.a.2},
we see from \eqref{hatphi-formula} that 
\bea\label{phi-ratio}
\frac{{\phi}(t,y)}{Ct_0^{1-\frac{\beta}{2}\eta}(t+t_0)^{\frac{\beta}{2}\eta-\frac{3}{2}} y e^{-\frac{\beta}{2}y-\tfrac{\beta^2}4t}e^{-\frac{y^2}{4(t+t_0)}}}=1+e^{\frac{y^2}{4(t+t_0)}}h(t,y).
\eea
Thanks to \eqref{e.a.1a}, we have 
\bea\label{h-estimate}
\limsup_{t\to\infty}\sup_{0\leq y\leq \sqrt{1+t}}|e^{\frac{y^2}{4(t+t_0)}}h(t,y)|=0.
\eea
Combining \eqref{phi-ratio} and \eqref{h-estimate}, we immediately ortain \eqref{e.a.2}.

To see \eqref{e.a.3}, we first observe that 
$$\max_{[0,\infty)} (y e^{-(\beta/2)y})= \left[y e^{-(\beta/2)y}\right]_{y=\tfrac2{\beta}}  =\tfrac{2}{\beta} e^{-1}> \max_{[4/\beta,\infty)} (y e^{-(\beta/2)y})= \tfrac4{\beta} e^{-2}.$$  Considering that $\|h(t,\cdot)\|_{L^\infty([0,\infty))} \to 0$ as $t\to\infty$ (thanks to \eqref{e.a.1a}), we deduce that  for any sufficiently large 
$t$, $\sup_{y\geq 0}\phi(t,y)$
must be achieved in the region $\{(t,y):~0\leq y\leq \sqrt{1+t}\}$. Therefore, \eqref{e.a.3} follows directly from \eqref{e.a.2}. 
This completes the proof. Finally, we can choose for the inequalities \eqref{e.a.1a}-\eqref{e.a.3} the same (large) constant $C_1$ large enough depending on the initial data $\phi_0$. 
\end{proof}

{\small 

}
\end{document}